\documentclass{amsart}
\usepackage[utf8]{inputenc}
\usepackage{amsmath}
\usepackage{amsthm}
\usepackage{amsfonts}
\usepackage{amssymb}
\usepackage{graphicx}%
\usepackage[usenames]{color}
\usepackage{verbatim}
\usepackage{tikz-cd}
\usepackage{bbm}
\usepackage{enumerate}
\usepackage{hyperref}
\usepackage{comment}
\usepackage{enumitem}

\providecommand{\U}[1]{\protect\rule{.1in}{.1in}}

\theoremstyle{plain}
\newtheorem{thm}{Theorem}[section]

\newtheorem{corollary}[thm]{Corollary}

\newtheorem{lemma}[thm]{Lemma}

\newtheorem{question}[thm]{Question}
\newtheorem{prop}[thm]{Proposition}

\theoremstyle{definition}

\newtheorem{rem}[thm]{Remark}
\newtheorem{deff}[thm]{Definition}
\newtheorem{example}[thm]{Example}

\DeclareMathOperator{\RR}{\mathbb{R}}
\def\N{{\mathbb N}}
\def\S{{\mathcal S}}
\def\I{{\mathrm I}}
\def\K{{\mathrm K}}

\def\Z{{\mathbb Z}}
\def\H{{\mathrm H}}

\def\C{{\mathcal C}}
\def \act{{G \curvearrowright X}}
\DeclareMathOperator{\const}{const}

\DeclareMathOperator{\mdim}{mdim}
\DeclareMathOperator{\mmdim}{mdim_{\Sigma, M}}
\DeclareMathOperator{\SC}{\S_{{c+}}}
\DeclareMathOperator{\h}{h}
\DeclareMathOperator{\leb}{Leb}

\DeclareMathOperator{\ord}{ord}
\DeclareMathOperator{\diam}{diam}

\DeclareMathOperator{\E}{E}
\DeclareMathOperator{\D_ms}{D^{\Sigma}_{md}}
\DeclareMathOperator{\ber}{b}
\DeclareMathOperator{\DD}{{\mathcal D}}
\DeclareMathOperator{\X_u}{X^{\Sigma}_{univ}}

\DeclareMathOperator{\R_u}{R^{\Sigma}_{univ}}

\newcommand{\Sym}{{\rm Sym}}
\newcommand{\Map}{{\rm Map}}

\newcommand{\shift}{{\rm shift}}

\newcommand{\Addresses}{{
  \bigskip

\hskip-\parindent F.~Garc\'ia-Ramos, \textsc{Physics Institute, Universidad Aut\'onoma de San Luis Potos\'i, M\'exico.\\
Faculty of Mathematics and Computer Science, Uniwersytet Jagielloński, Poland}\par\nopagebreak
  \textit{E-mail address}: \texttt{fgramos@conahcyt.mx}
  
 \medskip

\hskip-\parindent  Y.~Gutman, \textsc{Institute of Mathematics, 
    Polish Academy of Sciences, ul. Śniadeckich 8, 00-656 Warszawa,
    Poland.}\par\nopagebreak
  \textit{E-mail address}: \texttt{gutman@impan.pl}

}}

\title{Local mean
dimension theory for sofic group actions}
\author{Felipe García-Ramos and Yonatan Gutman}
\date{}

\begin{document}
\begin{abstract}

Using a local perspective, we introduce \textit{mean dimension pairs} and give sufficient conditions of when every non-trivial factor of a continuous group action of a sofic group $G$ has positive mean dimension. In addition we show that the mean dimension map is Borel, and that the set of subshifts with completely positive mean dimension of $[0,1]^G$, the full $G$-shift on the interval, is a complete coanalytic set in the set of all subshifts (hence not Borel). Our results are new even when the acting group is $\Z$.    
\end{abstract}
\date{\today}
\maketitle

 \section{Introduction}

 Mean dimension is an invariant of topological dynamical systems introduced by Gromov in  \cite{G}  and
studied and developed systematically by Lindenstrauss and Weiss in \cite{LW}. It has found many applications in various fields of mathematics such as 
topological dynamics (\cite{L99,LT14,Gut15Jaworski,gutman2017embedding,meyerovitch2019expansive,gutman2020embedding}), symbolic dynamics, (\cite{BDz2004,G11}), cellular automata (\cite{CSC10}), mathematical physics (\cite{mielke2002infinite,zelik2003attractors,efendiev2008finite,T2011,GutTakens,GQS18}), $C^\ast$-algebras (\cite{niu2014mean,HWZ15,elliott2017c,kerr2020almost}), information theory (\cite{lindenstrauss2018rate,gutman2018metric}) and complex analysis (\cite{G,G2008,matsuo2015brody}).

Starting from the fundamental work of Lindenstrauss and Weiss in \cite{LW}, exploring possible analogies and connections between the theory of mean dimension and the theory of entropy (both measurable and topological) has been a very fruitful idea. 
Heuristically, entropy measures the number of bits per time unit needed to describe a point in a system, whereas mean dimension measures the number of parameters per time unit. It is thus not surprising that the mean dimension of the $d$-cubical shift $(([0,1]^{d})^{\mathbb{Z}},\shift)$ is $d$ just as the entropy of the full shift on an alphabet with $a$ elements $\{0,1,\ldots, a-1\}^{\mathbb{Z}},\shift)$ is $\log a$. According to Krieger's generator theorem (\cite{k70}) an ergodic measure preserving systems with entropy $h<\infty$, has a generating partition with at most $e^{h}+1$ elements and this result is optimal. According to \cite{LT14, gutman2020embedding} a minimal topological dynamical system whose mean dimension is smaller than $\frac{N}{2}$ is embeddable in $(([0,1]^{N})^{\mathbb{Z}},\shift)$ and this result is optimal.

An important facet of the theory of entropy is the \textit{theory of local entropy} (see the surveys  \cite{glasner2009local,localsurvey}).
This theory was introduced by Blanchard \cite{blanchard1992fully, blanchard1993disjointness}, where his original motivation was to understand the topological analogs of $K$-systems, i.e., measure-preserving systems all of whose non-trivial factors have positive entropy. Specifically, he introduced the classes of \textit{completely positive entropy systems} (CPE) and \textit{uniform positive entropy systems} (UPE) as well as the related notion of \textit{entropy pairs}. 

The notion of entropy pairs and its variations have been used to further the study of many classes of dynamical systems such as null and tame systems \cite{kerr2009combinatorial}, $\infty$-step nilsystems  \cite{DDM13} and algebraic actions \cite{chung2015homoclinic,barbieri2022markovian}. Furthermore, local entropy theory has found applications in measurable dynamics (\cite{huang2006local,kerr2009combinatorial}),  functional analysis (\cite{kerr2009combinatorial,barbieri2022markovian}), topology \cite{darji2017chaos} and descriptive set theory \cite{darjilocal}. 

 We adopt the perspective of continuous actions of sofic groups. Following the breakthrough of Bowen in \cite{bowen2010measure}, Kerr and Li introduced topological sofic entropy in \cite{kerr2011entropy}, and Li sofic mean dimension in \cite{L12}, recovering several classical results. 
  
In this paper, we explore a theory analogous to the theory of local entropy theory in the context of mean dimension. We introduce the classes of \textit{completely positive mean dimension systems} (CPMD) and \textit{uniform positive mean dimension systems} (UPMD), as well as the concept of \textit{mean dimension pairs}. The theory we present exhibits intriguing parallels with the local theory of entropy, particularly in the realm of sofic group actions. This connection arises from the fact that, for sofic (non-amenable) group actions, entropy may strictly increase along factor maps (\cite[p. 222]{kerr2013combinatorial}), a phenomenon already observed under $\mathbb{Z}$-actions for mean dimension. We note that our contributions are new, even for actions of $\mathbb{Z}$.

\subsection*{Structure of the paper} In Section \ref{sec:2}, we recall basic notions from the theory of mean dimension and the theory of local entropy. In Section \ref{sec:cpmd} we introduce the classes of completely positive mean dimension systems (CPMD) and uniform positive mean dimension systems (UPMD). In Section \ref{sec:mdim_pairs}  we introduce and investigate the notion of mean dimension pairs and how it related to CPMD and UPMD. 
In Section \ref{sec:mdim_Borel} we introduce, for $G$ a sofic group and $A$ a compact metrizable space, $\S(A)$ the set of subactions of $A^G$, the full $G$-shift on $A$. We prove that the map which sends members of $\S(A)$, i.e.\ subactions of the full $G$-shift, to their mean dimension is a Borel map. In Section \ref{sec:descriptive_complexity} this last fact is used to prove that the set of subshifts with UPMD is Borel in the set of all subshifts. In addition, we show that the set of subsfhits of $[0,1]^G$ with CPMD is not Borel (in a very strong set), implying that it is a much more complicated property to check.

\subsection*{Acknowledgements} 
The authors are grateful to Hanfeng Li for pertinent comments and suggestions. 

Y.G. was partially supported by the National Science Centre (Poland) grant 2020/39/B/ST1/02329. F.G-R. was supported by the grant U1U/W16/NO/01.03 of the Strategic Excellence Initiative program of the Jagiellonian University.  

This work was partially supported by the Simons Foundation Award 663281 granted to the Institute of Mathematics of the Polish Academy of Sciences for the years 2021-2023.

The authors would like to thank the referee for comments that improved the readability of the paper.

\section{Preliminaries}\label{sec:2}

\subsection{Group actions}
For $n\in \N_+$, we define $[n]=\{1,...,n\}$. Throughout this paper, $X$ will always represent a \textbf{compact metrizable} space (with compatible metric $d$) and $G$ a countably infinite group with identity $e_G$. Denote the \textbf{diagonal} of $X\times X$ by $\Delta_X=\{(x,x)|\, x\in X\}$, and the family of non-empty finite subsets of $G$ by $\mathcal{F}(G)$.
\noindent
A (left) \textbf{action} of the group $G$ on $X$ will be represented by $\act$ or $(X,G)$. We will always assume that actions act by homeomorphisms of the topological space $X$.
The action of $g\in G$ on a point $x\in X$ will be written as $gx$. The open  $\epsilon$ ball around $x$ is denoted by $B_\epsilon(x)$. Its closure is denoted by $\overline{B}_\epsilon(x)$ (i.e.\ $\overline{B}_\epsilon(x)=\overline{B_\epsilon(x)}$).

\begin{deff}
Let $G \curvearrowright X$ and $G\curvearrowright Y$ be two actions. A function $\pi\colon X \to Y$ is \textbf{$G$-equivariant} if $g\pi(x)=\pi(gx)$ for every $g \in G$ and $x \in X$. In this situation, we write $\pi: \act \to G\curvearrowright Y$. If a $G$-equivariant function, $\pi:X\to Y$, is continuous and surjective, we say $G\curvearrowright Y$ is a \textbf{factor of} $\act$ and $\pi$ is a \textbf{factor map}. If $\emptyset\neq Z\subset X$ is a closed $G$-invariant subset of $X$, then we say $G\curvearrowright Z$ is a \textbf{subaction} of $\act$.

\end{deff}

\noindent
Let $\act$ be an action and $R\subset X\times X$ an equivalence relation. We say $R$ is $G$-\textbf{invariant} if $(gx,gy)\in R$ when $(x,y)\in R$. 
The following proposition is well known (e.g. see \cite[Appendix E.11-3.]{de2013elements}).
\begin{prop}\label{prop:factor-eq_rel}
    If $\pi\colon \act \to G\curvearrowright Y$ is a factor map, then $$R_\pi=\{(x,x')|\, \pi(x)=\pi(x')\}\subset X\times X$$ is a closed $G$-invariant equivalence relation. Conversely, if $Q\subset X\times X$ is a closed $G$-invariant equivalence relation then  $\pi\colon \act \to G\curvearrowright X/Q$ is a factor map where $G\curvearrowright X/Q$ is the induced (well-defined) action $gy=\pi(g\pi^{-1}(y))$. 
\end{prop}

We say $(X,T)$ is a \textbf{topological dynamical system (t.d.s.)} if $X$ is a compact metrizable space and $T:X\to X$ a homeomorphism. There is a standard identification between topological dynamical systems and  $\Z$-actions.

\subsection{Sofic groups} (Gromov; \cite{gromov1999endomorphisms})
For $n\in \N$ we write $\Sym(n)$ for the group of permutations of $[n]$. 
A group $G$ is \textbf{sofic} if there exists a sequence $\left\{ n_{i}\right\}_{i=1}^{\infty}$ of positive integers which go to infinity, and a sequence $\Sigma=\{s_i \colon G\rightarrow \Sym(n_i) \}_{i=1}^{\infty}$ that satisfies
\begin{align*}
\lim_{i\rightarrow\infty} \frac{1}{n_i} \left\vert \left\{ v\in [n_{i}]
: s_{i}(gg')v=s_{i}(g)s_{i}(g')v\right\}  \right\vert  &
=1 \mbox{ for every } g,g'\in G \text{, and}\\
\lim_{i\rightarrow\infty} \frac{1}{n_i}\left\vert \left\{  v\in[n_i]
: s_{i}(g)v\neq s_{i}(g')v\right\}  \right\vert  & =1 \mbox{ for every } g\neq g'\in G.
\end{align*}
In this case we say $\Sigma$ is a \textbf{sofic approximation sequence} for $G$. 
Amenable groups and residually finite groups are sofic. 
\subsection{Operations on open covers}

Let $\alpha$ and $\beta$ be finite open covers of $X$. 
 We say that $\beta$ \textbf{refines} $\alpha$, denoted $\beta \succ \alpha$, if every member of $\beta$ is contained in a member of $\alpha$.
 The \textbf{join} of $\alpha$ and $\beta$ is defined as $\alpha\vee \beta=\{A\cap B|\ A\in \alpha,B\in \beta\}$. 
 Let $n\in \N$. Denote by $\alpha^{[n]}$ the finite open cover of $X^{[n]}$, where 
$$X^{[n]}=\underbrace{X\times\cdots \times X}_{n \textrm{ times}}$$
consisting of sets of the form $U_1\times U_2\times \dots \times U_n$ for $U_1, \dots, U_n\in \alpha$.

\subsection{Sofic topological entropy} (Kerr and Li; \cite{kerr2011entropy,kerr2016ergodic})
We will use the notation 
$$d^2 (x,y)=\big (d(x,y)\big )^2.$$

   Let $\act$ be an action, $F\in \mathcal{F}(G)$, $\delta>0,$ $n\in\mathbb{N}$, and $$s\colon G\to
	\Sym(n).$$ Define $\Map(d,F,\delta,s)$ as the set of all maps
	$\varphi\colon [n]  \to  X$ such that
	\[
	\left(  \frac{1}{n}\sum_{v=1}^{n}d^{2}(\varphi(s_g(v)),g\varphi
	(v))\right)  ^{1/2}\leq\delta \ \mbox{for every }g \in F, 
	\]
	and  $\Map'(d,F,\delta,s,X)$ as the set of all maps
	$\varphi\colon [n]  \to  X$ such that
	\[
	\left(  \frac{1}{n}\sum_{v=1}^{n}d^{2}(\varphi(s_g(v)),g\varphi
	(v))\right)  ^{1/2}< \delta \ \mbox{for every }g \in F. 
	\]

Note that $\Map(d, F, \delta, s)$ is a closed subset of $X^{[n]}$, whereas $\Map'(d, F, \delta, s)$ is an open subset of $X^{[n]}$. 

\noindent

Let $\alpha$ be a non-empty finite open cover of $X$ and $\Sigma=\{s_i \colon G\rightarrow \Sym(n_i) \}_{i=1}^{\infty}$ a sofic approximation sequence for $G$. Let $Y\subset X^\ell$ for some $\ell\in \N$ and $\beta$ a finite open cover of $X^\ell$. Denote by $N(\beta,Y)$  be the minimal cardinality of a subset of $\beta$ which covers $Y$ (with the convention  $N(\beta,\emptyset)=0$). Define (with the convention  $\log 0=-\infty$) 
$$h_{F, \delta, \Sigma} (X,G, \alpha)= \limsup_{i\rightarrow \infty} \frac{1}{n_i} \log N\left(\alpha^{[n_i]},\Map(d,F,\delta,s_i) 
\right),$$

and 
$$h_{\Sigma} (X,G, \alpha)= \inf_{F\in \mathcal{F}_G} \inf_{\delta> 0} h_{F, \delta} (G, \alpha).
 $$
We say $h_{\Sigma} (X,G, \alpha)$ is the \textbf{sofic topological  entropy} of $\alpha$ for $\act$. Define the \textbf{sofic topological  entropy} of $\act$ as
$$h_{\Sigma} (X,G)= \sup_{\alpha} h_{\Sigma} (X,G, \alpha),$$
where $\alpha$ ranges over finite open covers of $X$. 

\begin{rem}
    There exist actions of sofic groups where different approximation sequences give different values of entropy \cite[Theorem 4.1]{bowen2020examples} (also see \cite[Theorem 1.1]{airey2022topological}).
\end{rem}

\begin{rem}\label{rem:top_ent_metric}(\cite[Definition 4.1]{L12} and \cite[Definition 10.22]{kerr2016ergodic})
    It is possible to define the sofic topological  entropy with respect to a continuous pseudometric $\rho$ (in particular $\rho=d$) on $X$ by defining quantities $h^{\epsilon}_{F, \delta, \Sigma} (X,G, \rho)$,  $h^{\epsilon}_{ \Sigma} (X,G, \rho)$ and finally 
    $$
    h_{ \Sigma} (X,G)=\lim_{\epsilon\rightarrow 0} h^{\epsilon}_{ \Sigma} (X,G, \rho).
    $$
    \end{rem}

\begin{rem}
  If $G$ is an amenable group, then the sofic topological entropy of an action $\act$ with respect to any approximation sequence coincides with the classical topological entropy \cite{kerr2013soficity}.   
\end{rem}
\subsection{Sofic mean dimension}\label{subsec:sofic mdim} (Li; \cite{L12}) 

\begin{deff}
Let $\alpha$ be a finite open cover of $X$. Denote
$$
\ord(\alpha)=(\max_{x\in X}\sum_{U\in \alpha}1_{U} (x))-1 \text{, and}
$$

$$
\DD(\alpha)=\min_{\beta \succ \alpha}\ord(\beta).
$$
\end{deff}
Note the trivial inequality
$$\DD(\alpha)\leq\ord(\alpha)<\infty$$
\begin{deff}\label{def:D_subset}
Let $Y\subset X$ be a (not necessarily closed) non-empty subset of $X$. Let $\alpha$ be a finite open cover of $X$. Denote
$$
\ord(\alpha|_{Y})=(\max_{x\in Y}\sum_{U\in \alpha}1_{U} (x))-1 \text{, and}
$$

\begin{equation} \label{eq:D_alpha_Y}
\DD(\alpha|_{Y})=\min_{\beta \succ \alpha}\ord(\beta|_{Y}).    
\end{equation}

\noindent
where the minimum is taken over finite open covers $\beta$ of $X$ (not $Y$) such that $\beta \succ \alpha$.
\noindent
In addition we define
$$
\ord(\alpha|_{\emptyset})=-\infty.
$$
\end{deff}

\begin{rem}\label{closed_case}
    If $Y$ is closed, then taking in definition \eqref{eq:D_alpha_Y} the minimum over finite open covers  $\beta$ of $Y$ such that $\beta \succ \alpha|_{Y}$ will give the same quantity. However, note that when $Y$ is not closed, this might fail. 
\end{rem}
\begin{rem}\label{rem:monotonicity}
It is clear that if $\beta \succ \alpha$ then $\DD(\beta)\geq \DD(\alpha).$
\end{rem}
\noindent
The following lemma appeared in \cite[Corollary 2.5]{LW}.
\begin{lemma}
[Lindenstrauss and Weiss]
\label{lem:subadd}

Let $\alpha$ and $\beta$ finite open covers of $X$. Then $$\DD (\alpha \vee \beta)\leq \DD (\alpha)+ \DD (\beta).$$ 
\end{lemma}
\begin{prop}\label{prop:mdim_func}
 Let $X,Y$ be compact metrizable spaces and $f:X\rightarrow Y$ a continuous surjective map. Suppose that  $\alpha$ is a finite open cover of $Y$. Then
 $$ \DD (f^{-1}(\alpha))\leq  \DD (\alpha). $$
\end{prop}
\begin{proof}
    Note that  if $\beta \succ \alpha$ then  $f^{-1}(\beta) \succ f^{-1}(\alpha)$ and $\ord(f^{-1}(\beta))=\ord(\beta)$.
\end{proof}
\begin{prop}\label{prop:mdim_nD}
   $$ \DD(\alpha^{[n]})\leq n \DD (\alpha). $$
\end{prop}
\begin{proof}
    Let $\pi_i: X^{[n]}\rightarrow X$ be the projection on the $i$-th factor. Note:
    $$\alpha^{[n]}=\bigvee_{i=1}^{n} \pi_i^{-1}(\alpha).$$
The result now follows from  Lemma \ref{lem:subadd} and Proposition \ref{prop:mdim_func}.
    
\end{proof}
\noindent
In \cite{L12} Li defined mean dimension for sofic group actions. Let us give the details.

\noindent

Define
$$\DD(\alpha, d, F, \delta, s)=\DD(\alpha^{[n]}|_{\Map(d, F, \delta, s)})\text{ and}$$ 
\begin{equation}\label{eq:D'}
\DD'(\alpha, d, F, \delta, s,X)=\DD(\alpha^{[n]}|_{\Map'(d, F, \delta, s,X)}).
\end{equation}

\noindent If $\Map(d, F, \delta, s)=\emptyset$, then set $\DD(\alpha, d, F, \delta, s)=-\infty$ and similarly for $\DD'(\alpha, d, F, \delta, s)$. Clearly $\DD'(\alpha, d, F, \delta, s,X)\leq \DD(\alpha, d, F, \delta, s,X)$.
 
\noindent
Let $F\in \mathcal{F}(G)$ and
$\delta > 0$. For a finite open cover $\alpha$ of $X$ and a sofic approximation sequence for $G$, $\Sigma=\{s_i \colon G\rightarrow \Sym(n_i) \}_{i=1}^{\infty}$, we define
\begin{align*}
\mdim_\Sigma(\alpha, d ,F, \delta ) &=
\varlimsup_{i\to\infty} \frac{\DD(\alpha, d, F, \delta, s_i)}{n_i},\\
\mdim_\Sigma(\alpha ,F ) &= \inf_{\delta > 0} \mdim_\Sigma(\alpha, d ,F, \delta ),\text{ and}\\
\mdim_\Sigma(\alpha)&= \inf_{F'\in \mathcal{F}(G)} \mdim_\Sigma(\alpha,F').
\end{align*}

\noindent Note that if $\Map (d ,F,\delta ,s_i )$ is empty for all sufficiently large $i$, then
$$\mdim_\Sigma(\alpha, d ,F, \delta ) = -\infty.$$
Define the {\bf sofic mean dimension (with respect to $\Sigma$) } of $\act$  as
$$ \mdim_\Sigma (X,G) = \sup_{\alpha} \mdim_\Sigma(\alpha)$$
for $\alpha$ ranging over  finite open covers of $X$.

\begin{rem}
\label{rem:-infty}
Note that $\mdim_{\Sigma}(X,G)=-\infty$ if and only if $\h_{\Sigma}(X,G)=-\infty$. There exist examples with this property \cite[Subsection 3.1.2]{bowen2018brief}.
\end{rem}

\begin{lemma}
Let $\act$ be an action and $\alpha$ a finite open cover of $X$. It holds that $\mdim_\Sigma(\alpha, d ,F, \delta )\leq \DD(\alpha)$ and $\mdim_\Sigma(\alpha)\leq \DD(\alpha)$. 
\end{lemma}

\begin{proof}
By Proposition \ref{prop:mdim_nD} it holds that $\DD(\alpha^{[n]}|_{\Map(d, F, \delta, s)})\leq \DD(\alpha^{[n]})\leq n\DD(\alpha).$
This implies that 
\[
\varlimsup_{i\to\infty} \frac{\DD(\alpha, d, F, \delta, s_i)}{n_i}\leq \DD(\alpha).
\]
The other inequality follows from the definition. 
\end{proof}
\begin{example}
\label{example_d}
Let $\Sigma$ be a sofic approximation of $G$, $d\in \N\cup \{\infty\}$, and $G \curvearrowright I^d$ the left action induced by the shift. It holds that $\mdim_{\Sigma}([0,1]^d,G)=d$ \cite[Theorem 7.1]{L12}.
\end{example}

\begin{rem} \label{rem:notation}
Assume $Y$ is a (strict) subaction of $X$ and $\beta$ is an open cover of $Y$, then one may write
$\mdim_\Sigma(Y,\beta)$ or $\mdim_\Sigma(Y,G,\beta)$ in order to emphasize that all operations are within $Y$.
If $\alpha$ is an open cover of $X$, then one may write interchangeably $\mdim_\Sigma(Y,\alpha|_{Y})$ or $\mdim_\Sigma(\alpha|_{Y})$ as the context is clear.

\end{rem}

\begin{rem} \label{mdim_conj}

The quantities
$\mdim_\Sigma(\alpha ,F )$, $\mdim_\Sigma(\alpha)$ and $\mdim_\Sigma(X)$ do not depend on the choice of the compatible metric $d$ (\cite{L12}). In particular, $\mdim_\Sigma(X,G)$ is invariant under conjugacy.
\end{rem}

\begin{rem} \label{R-mean top dim1}
Note that $\mdim_\Sigma(\alpha, d ,F, \delta )$ decreases when $\delta$ decreases and $F$ increases. Thus, in the definitions of $\mdim_\Sigma(\alpha ,F )$ and $\mdim_\Sigma(\alpha)$ one can also replace $\inf_{\delta>0}$ and $\inf_F$ by $\lim_{\delta\to 0}$ and $\lim_{F\to \infty}$ respectively, where $\lim_{F\to \infty}$  denotes the net limit associated with the partial order on $\mathcal{F}(G)$ defined by $F_1\le F_2$ if and only if  $F_1\subseteq F_2$. Similarly defining the partial order on  $\mathcal{F}(G)\times \RR_+$ by $(F, \delta)\ge (F', \delta)$ if and only if $F\supseteq F'$ and $\delta\le \delta'$,
one may write
$$ \mdim_\Sigma(\alpha)= \lim_{(F, \delta)\to (\infty,0)} \mdim_\Sigma(\alpha, d ,F, \delta).$$
\end{rem}
\noindent
The definition of $\DD'(\alpha, d, F, \delta, s)$ will be used in Section \ref{sec:mdim_Borel} through the following lemma. 

\begin{lemma}\label{lem:D'}
     Let $\act$ be an action and $F\in \mathcal{F}(G)$. For a finite open cover $\alpha$ of $X$ and $\Sigma=\{s_i \colon G\rightarrow \Sym(n_i) \}_{i=1}^{\infty}$ a sofic approximation sequence for $G$, it holds
     {
      $$
\mdim_\Sigma(\alpha,F) = 
\lim_{\delta \rightarrow 0} 
\varlimsup_{i\to\infty} \frac{\DD'(\alpha, d, F, \delta, s_i)}{n_i}.
    $$
    }
\end{lemma}
\begin{proof}
 Note that for all $\delta>0$ and $i\in \N$,
 $$
 \Map(d,F,\delta/2,s_i,X) \subset \Map'(d,F,\delta,s_i,X)\subset \Map(d,F,\delta,s_i,X).
  $$
  As by Remark \ref{R-mean top dim1} one may replace $\inf_{\delta > 0}$ by $\lim_{\delta \rightarrow 0}$; the proof is completed.
\end{proof}

 The following lemma is probably known. 
\begin{lemma}
\label{lem:basic}
     Let $\act$ an action, $\Sigma$ a sofic approximation sequence for $G$ and $\alpha,\beta$ finite open covers of $X$. The following holds:
 \begin{enumerate}
     \item If $\alpha \succ \beta$ then $\mdim_{\Sigma}(\alpha)\geq \mdim_{\Sigma}(\beta)$.
     \item $\mdim_{\Sigma} (\alpha \vee \beta)\leq \mdim_{\Sigma} (\alpha)+ \mdim_{\Sigma} (\beta).$
 \end{enumerate}
\end{lemma}
\begin{proof}
   Let  $F\in \mathcal{F}(G)$, $\delta>0,$ $n\in\mathbb{N}$, and $\Sigma\ni s\colon G\to \Sym(n).$  To prove $(1)$ note that as $\alpha \succ \beta$ then $\alpha^{[n]} \succ \beta^{[n]}$ and in particular $\alpha^{[n]}|_{\Map(d, F, \delta, s)}\succ\beta^{[n]}|_{\Map(d, F, \delta, s)}$. By Remark \ref{rem:monotonicity} it holds that $\DD(\alpha^{[n]}|_{\Map(d, F, \delta, s)})\geq \DD(\beta^{[n]}|_{\Map(d, F, \delta, s)})$, whenever $\alpha \succ \beta$.  Thus
$\mdim_{\Sigma}(\alpha)\geq \mdim_{\Sigma}(\beta)$.
    
    To prove $(2)$ note that from Lemma \ref{lem:subadd} and the equality $(\alpha \vee \beta)^{[n]}=\alpha^{[n]}\vee \beta^{[n]}$ one has 
  \begin{align*}
\DD((\alpha \vee \beta)^{[n]}|_{\Map(d, F, \delta, s)}) &= \DD(\alpha^{[n]}\vee \beta^{[n]}|_{\Map(d, F, \delta, s)}) \\
&\leq \DD(\alpha^{[n]}|_{\Map(d, F, \delta, s)}) + \DD(\beta^{[n]}|_{\Map(d, F, \delta, s)})
\end{align*}
Thus $\mdim_{\Sigma} (\alpha \vee \beta)\leq \mdim_{\Sigma} (\alpha)+ \mdim_{\Sigma} (\beta).$

\end{proof}

\begin{rem}
When $G$ is an amenable group the mean dimension of an action $\act$ does not depend on the choice of the sofic approximation sequence and coincides with the classical notion of mean dimension \cite{L12,jin2023sofic}, so we omit writing $\Sigma$ in this situation. To the best of our knowledge, it is an open question if there exists a group action of a sofic group where the sofic mean dimension depends on the choice of the approximation sequence.  
\end{rem}

We say $\act$ is a \textbf{trivial action} if for all $x\in X$ and $g\in G$, $gx=x$.
The following fundamental result follows from known results. 
\begin{prop}
\label{thm:trivial_htop_0}
   Let $\Sigma$ be a sofic approximation sequence for $G$,  and $\act$ a trivial action. Then  $h_{ \Sigma} (X,G)= 0$.
\end{prop}

\begin{proof}
   Let  $F\in \mathcal{F}(G)$, $\delta>0,$ $n\in\mathbb{N}$, and $\Sigma\ni s\colon G\to \Sym(n).$ Note that any fixed map of the form $\varphi\colon [n]  \to  X$, $\varphi\equiv x_0\in X$, trivially obeys
   $$
\frac{1}{n}\sum_{v=1}^{n}d^{2}(\varphi(s_g(v)),g\varphi
	(v))=\frac{1}{n}\sum_{v=1}^{n}d^{2}(x_0,x_0)=0 \ \mbox{for every }g \in F. 
   $$
   
  Thus $\varphi\in \Map(d,F,\delta,s)$ and $\Map(d,F,\delta,s)\neq \emptyset$. Conclude $h_{\Sigma} (X,G)\geq 0$. 
  Every trivial action is distal, that is $\inf_{g\in G}d(gx,gy)>0$ for all $x\neq y\in X$, hence $h_{ \Sigma} (X,G)\leq 0$ (\cite[Corollary 8.5]{kerr2013combinatorial}). 
  
\end{proof}

\begin{prop}\label{thm:trivial_mdim_0}
   Let $\Sigma$ be a sofic approximation sequence for $G$,  and $\act$ a trivial action. Then  $\mdim_\Sigma (X,G)=0$.
\end{prop}

\begin{proof}
The proof of Proposition \ref{thm:trivial_htop_0} indicates $\Map(d,F,\delta,s)\neq \emptyset$. Thus $\mdim_\Sigma (X,G)\geq 0$. 
 Again using Proposition \ref{thm:trivial_htop_0} and Remark \ref{rem:top_ent_metric}, we obtain that $h^{\epsilon}_{ \Sigma} (X,G, d)$ is uniformly bounded from above for small enough $\epsilon>0$. Thus the \textit{sofic metric mean dimension} (\cite[Definition 4.1]{L12}) of  $\act$ defined as
  $$
  \mmdim(X,d)=\liminf_{\epsilon\rightarrow 0}{-\frac{h^{\epsilon}_{ \Sigma} (X,G, d)}{\log \epsilon}}
  $$
 equals zero. By \cite[Theorem 6.1]{L12}, $0\leq \mdim_\Sigma (X,G)\leq \mmdim(X,d)=0$. We conclude  $\mdim_\Sigma (X,G)=0$ as desired.    
\end{proof}

Let $\act$ and action, $F\in \mathcal{F}(G)$ and $\alpha$ a finite open cover of $X$. We define
	
	$$\alpha_F = \bigvee_{g \in F}g^{-1}\alpha.$$
 In particular, given a t.d.s., it holds that $\alpha_{[n]}=\bigvee_{i=0}^{n-1} T^{-i}\alpha. $
We now define the classic notion of  mean dimension introduced by Gromov (\cite{G}) and studied  by Lindenstrauss and Weiss (\cite{LW}).

\begin{deff}Let $(X,T)$ be a t.d.s. and $\alpha$ a finite open cover of $X$. Define the \textbf{mean dimension of $(X,T)$ with respect to $\alpha$} as$$ \mdim (X,T,\alpha)=\lim_{n\rightarrow \infty}\frac{\DD( \alpha_{[n]})}{n}.$$ The \textbf{mean dimension} of $(X,T)$ is the supremum of $\mdim (X,T,\alpha)$ over all finite open covers. \end{deff}

The previous notion was naturally generalized for actions of amenable groups. 

\begin{deff}
\label{def:amenable}
Let $G$ be an amenable group, $\act$ an action, and $\alpha$ a finite open cover of $X$. We define the \textbf{mean dimension of $\act$ with respect to $\alpha$} as
$$ \mdim (X,G,\alpha)=\lim_{F}\frac{\DD( \alpha_{F})}{|F|},$$
where the limit is taken as $F$ becomes more and more invariant, that is, there is $L\in \mathbb{R}$ so that for every $\epsilon>0$, there exist $A\in \mathcal{F}(G)$ and $\delta > 0$ such that $$|\frac{\DD( \alpha_{F})}{|F|}- L| < \epsilon$$
for every $F\in \mathcal{F}(G)$ with $|{g \in F : Ag \subseteq F}| \geq (1-\delta)|F|$. The existence of the limit follows from \cite[Theorem 6.1]{LW}.
The \textbf{mean dimension of }$\act$ is the supremum of $\mdim (X,G,\alpha)$ over all finite open covers and is denoted by $\mdim(X,G)$. 

\end{deff}

For a proof of the following result see \cite[Theorem 3.1]{L12}). 
\begin{thm}
[Li]
\label{thm:amenable}
    Let $G$ be an amenable group, $\Sigma$ a sofic approximation sequence for $G$ and $\act$ an action. 
    Then $\mdim (X,G)= \mdim_{\Sigma} (X,G)$. 
\end{thm}

It is known that if an action of an amenable group has positive mean dimension, then it must have positive topological entropy \cite{LW}. It is not known if the result holds at the local level (the following question is also open for actions of $\Z$). 

\begin{question}
\label{ques:local}
 Let $\act$ an action and $\Sigma$ a sofic approximation sequence for $G$. Suppose $\alpha$ is a finite open cover of $X$ with $$\h_{\Sigma}(X,T,\alpha)=0.$$ Is 
it true that $\mdim_{\Sigma} (\alpha)=0?$
\end{question}

\subsection{The universal zero mean dimension factor}

Lindenstrauss introduced \textit{a} universal zero mean dimension factor of topological dynamical systems, this notion can be generalized for actions of sofic groups.

Let $\Sigma$ be a sofic approximation sequence for $G$, $\act$ an action, and $\R_u$ the biggest $G$-invariant closed equivalence relation which is contained in all $G$-invariant closed equivalence relations in $X$ that induce quotients with non-positive sofic mean dimension. 
Formally
\[
\R_u=\bigcap \{R:R=\overline{R} \text{ is a $G$-invariant eq. relation with }\mdim_{\Sigma}(G,X/ R)\leq 0\}.
\]
Note that there exists at least one $G$-invariant closed equivalence relation in $X$ that induces a quotient with non-positive sofic mean dimension, namely $R=X\times X$. 
Furthermore, if $\mdim_\Sigma (X,G)\leq 0$, then $\R_u=\Delta_X$.
Define $\X_u=X/\R_u$.

Using Proposition \ref{prop:factor-eq_rel} we conclude that
\[
\R_u=\bigcap\{R_{\pi}:\pi \text{ is a factor map and }\mdim_{\Sigma}(G,\pi(X))\leq 0\}.
\]
From this formula the next observation follows. 
\begin{lemma}
\label{lem:univ_sep}
Let $\Sigma$ be a sofic approximation sequence for $G$ and $\act$ an action. Then $(x,y)\in \R_u$ if and only if $(x,y)$ cannot be separated by a factor with non-positive sofic mean dimension. 
\end{lemma}

\noindent
When $\mdim_\Sigma (X,G)\geq 0$, the map $\pi\colon \act \rightarrow G\curvearrowright \X_u$ is called the \textbf{universal zero sofic mean dimension factor}\footnote{The reason one requires $\mdim_\Sigma (X,G)\geq 0$ is to be able to prove $\mdim_{\Sigma}(\X_u,G)=0$ (see Proposition \ref{prop:uzmdf}). Otherwise it might hold  $\mdim_{\Sigma}(\X_u,G)=-\infty$.} of $\act$. The following result is a generalization of \cite[Proposition 6.12]{L99}, for a proof see \cite[Proposition 2.14]{L12}.  
\begin{prop}
[Lindenstrauss; Li]
\label{prop:uzmdf}
Let $\Sigma$ be a sofic approximation sequence for $G$ and $\act$ an action with $\mdim_\Sigma (X,G)\geq 0$. Then $$\mdim_{\Sigma}(\X_u,G)=0$$ and every factor map onto a zero sofic mean dimension system $$\psi: \act \rightarrow G\curvearrowright Y$$ factors through $\pi$, i.e.\ there exists a factor map $\phi:G\curvearrowright \X_u \rightarrow G\curvearrowright Y$  such that $\phi\circ\pi=\psi$. 
\end{prop}

\begin{rem}
Let $G\curvearrowright Y$ be a factor of $\act$. If $\mdim_{\Sigma}(X,G)\geq 0$ then $\mdim_{\Sigma}(Y,G)\geq 0$ \cite[page 9]{L12}.    
\end{rem}

\subsection{Local theory of entropy}
 Blanchard \cite{blanchard1992fully, blanchard1993disjointness} introduced the following notions for topological dynamical systems. We state the sofic entropy versions. 
 
 \begin{deff} Let $\act$ an action and $\Sigma$ a sofic approximation sequence for $G$. We say that $\act$ has \textbf{completely positive sofic entropy (sofic CPE)} if every non-trivial factor has positive sofic topological entropy with respect to $\Sigma$. 
     
 \end{deff}
\begin{deff}\label{def:standard}
An open cover of $X$ is \textbf{standard} if it is composed of two non-dense open sets. An open cover $(U,V)$ \textbf{distinguishes} $(x,y)\in X\times X$ if $x\notin \overline{V}$ and $y\notin \overline{U}$. Such a cover is always standard. Conversely, every standard cover distinguishes some $(x,y)\in X\times X$.
\end{deff}

 \begin{deff}  Let $\act$ be an action and $\Sigma$ a sofic approximation sequence for $G$. We say that $\act$ has \textbf{uniform positive sofic entropy (sofic UPE)} if 
 $$
 h_{\Sigma}(X,G,\alpha)>0
 $$
 for every standard open cover, $\alpha$, of $X$.
     
 \end{deff}

\begin{deff} 
Let $\act$ be an action and $\Sigma$ a sofic approximation sequence for $G$. We say that $(x,y)\in X\times X \setminus \Delta_X$ is a \textbf{sofic entropy pair} if for every standard open cover $\alpha$ which distinguishes $(x,y)$ it holds that $h_{\Sigma}(X,G,\alpha)>0$. Denote the sofic entropy pairs by $\E_{\Sigma}(X,G)$. 
\end{deff}

The following result was obtained in \cite[Proposition 4.16 and Remark 4.4]{kerr2013combinatorial}. 
\begin{thm}
[Kerr and Li]
\label{thm:sofic IE properties} Let $G$ be a sofic group, $\Sigma$ a sofic approximation sequence for $G$, and $\act$, $G\curvearrowright Y$ continuous group actions. Then
    \begin{enumerate}
  \item $h_{\Sigma}(X,G)>0$ if and only if $\E_{\Sigma}(X,G)\neq \emptyset$. 
        \item Let $\pi\colon \act \to G\curvearrowright Y$ be a factor map. If $(x,y)\in \E_{\Sigma}(X,G)$ and $\pi(x)\neq \pi (y)$ then $(\pi(x),\pi(y))\in \E_{\Sigma}(Y,G).$
        
    \end{enumerate}
    
\end{thm}

The following result is a natural generalization of a result presented in \cite{blanchard1993disjointness} within the context of sofic entropy. It is known, but it does not appear explicitly in the literature.
\begin{corollary}
    Let $G$ be a sofic group and $\Sigma$ a sofic approximation sequence for $G$. An action $\act$ has sofic CPE if the smallest $G$-invariant closed equivalence relation that contains $\E_{\Sigma}(X,G)$ is $X\times X$. 
\end{corollary}

\begin{proof}
If $\act$ does not have sofic CPE then there exists a closed $G$-invariant equivalence relation $R\neq X\times X$ so that $h_{\Sigma}(X/R,G)\leq 0$. We claim that the inequality implies that $\E_{\Sigma}(X,G)\subset R$. Otherwise by Theorem \ref{thm:sofic IE properties} (2), $E_{\Sigma}(X/R,G)\neq \emptyset$, and by Theorem \ref{thm:sofic IE properties} (1), $h_{\Sigma}(X/R,G)> 0$.
This concludes the proof. 
\end{proof}

In the sequel we prove analogous results in the context of mean dimension (see Corollary \ref{cor:zmd}, Proposition \ref{prop:factor1}, and Theorem \ref{thm:eqrelation}).
\section{Systems with completely/uniform positive mean dimension}\label{sec:cpmd}

\begin{deff}\label{def:CPMD}
Let $\act$ be an action and $\Sigma$ a sofic approximation sequence for $G$. We say that $\act$ has \textbf{completely positive sofic mean dimension (sofic CPMD)} if every non-trivial factor has positive sofic mean dimension. 
\end{deff}

The following result is a consequence of Lemma \ref{lem:univ_sep}. 
\begin{corollary}
\label{cor:cpmd_iff_umdz_trivial}

Let $\Sigma$ be a sofic approximation sequence for $G$ and $\act$ an action so that $|X|\geq 2$ and $\mdim_\Sigma (X,G)\geq 0$. The universal zero sofic mean dimension factor is trivial, equivalently $\R_u=X\times X$,  if and only if $\act$ has sofic CPMD. 
\end{corollary}

\begin{deff}\label{def:UPMD}
Let $G$ be a sofic group. An action $\act$ has \textbf{uniform positive sofic mean dimension (sofic UPMD)} if $\mdim_{\Sigma} (\alpha)>0$ for every standard open cover $\alpha$.
\end{deff}

According to Corollary \ref{cor:upmd implies cpmd} in the sequel, every non-trivial action with uniform positive sofic mean dimension system has completely positive sofic mean dimension. The opposite direction does not hold, as the following example demonstrates. 

\begin{deff}
  Let $G$ be an amenable group and $\act$ an action.  We say $\act$ has \textbf{completely positive mean dimension (CPMD)} if every non-trivial factor has positive mean dimension. We say $\act$ has \textbf{uniform positive mean dimension (UPMD)} if $\mdim (\alpha)>0$ for every standard open cover $\alpha$.
\end{deff}

\begin{rem}
 Let $G$ be an amenable group and $\act$ an action. Since sofic mean dimension coincides with mean dimension (Theorem \ref{thm:amenable}) it holds that $\act$ has sofic CPMD if and only if it has completely positive mean dimension. We do not know, however, if sofic UPMD depends on the choice of the sofic approximation sequence.
    
\end{rem}
\begin{example}(\textit{Example of a system with CPMD but not with UPMD.})
Let $X=[0,1]^{\Z}\cup [1,2]^\Z$ equipped with the shift map $\sigma\colon X\rightarrow X$. Let $x\in [0,1]^{\Z}\setminus \{\vec{1}\}$ and $y\in [1,2]^{\Z}\setminus \{\vec{1}\}$ and fix a standard cover  $\alpha=(U,V)$  which distinguishes $(x,y)$  such that $[0,1]^{\Z}\subset U$ and $[1,2]^\Z\subset V$. Note that for all $n$, $[0,1]^{\Z}\subset\bigcap_{i=0}^{n-1} \sigma^{-i}U$ and $[1,2]^{\Z}\subset\bigcap_{i=0}^{n-1} \sigma^{-i}V$. Moreover, $\{\bigcap_{i=0}^{n-1} \sigma^{-i}U,\bigcap_{i=0}^{n-1} \sigma^{-i}V\}\subset \alpha_{[n]}$ and their union equals $X$. This implies $\{\bigcap_{i=0}^{n-1} \sigma^{-i}U,\bigcap_{i=0}^{n-1} \sigma^{-i}V\}\succ \alpha_{[n]}$. Thus, $\DD (\alpha_{[n]})=1$ for all $n$ which implies  $\mdim (X,\sigma,\alpha)=0$. Conclude $(X,\sigma)$ does not have uniform positive mean dimension. However, by Proposition \ref{cor:cpmd iff A countable} in the sequel, $(X,\sigma)$ has completely positive mean dimension.
\end{example}

\section{Mean dimension pairs}\label{sec:mdim_pairs}

\begin{deff}
\label{def:mdpairs}
 Let $\act$ be an action and $\Sigma$ a sofic approximation sequence for $G$. A pair $(x,y)\in X\times X$ is said to be a \textbf{sofic mean dimension pair} if for every standard open cover, $\alpha$, which distinguishes $(x,y)$, it holds  $\mdim_{\Sigma} (\alpha)>0$. Denote the set of mean dimension pairs by $\D_ms(X,G)$. Note that $\D_ms (X,G)\subset (X\times X)\setminus \Delta_X$.
\end{deff}

\begin{prop}
\label{prop:upmd}
Let $\Sigma$ be a sofic approximation sequence for $G$. An action $(X,G)$ has sofic UPMD if and only if $\D_ms(X,G)=X^2\setminus \Delta_X$.
\end{prop}
\begin{proof}
    Assume $(X,G)$ has sofic UPMD and let $(x,y)\in X^2\setminus \Delta_X$. Clearly for every standard open cover, $\alpha$, which distinguishes $(x,y)$ it holds  $\mdim_{\Sigma} (\alpha)>0$. Conversely every standard open cover distinguishes some $(x,y)\in X^2\setminus \Delta_X$, so if $\D_ms(X,G)=X^2\setminus \Delta_X$, then for all  standard open covers $\alpha$ it holds  $\mdim_{\Sigma} (\alpha)>0$.
\end{proof}
\subsection{Existence and properties of mean dimension pairs}\label{subsec:prop_mdim_pairs}

Given a group $G$ and a compact metrizable space $A$, $A^{G}$ is equipped with the product topology (which is metrizable). Define the left shift action of $G$ on $A^G$  by $(gx)_h=x_{g^{-1}h}$ for all $x\in Z^G$ and $g, h\in G$.
The following result follows from the proof of \cite[Theorem 7.5]{L12}.

\begin{prop}
[Li]
\label{prop:AG_UPMD}
Let $A$ be a path-connected compact metrizable space, $G$ a sofic group. The left shift action of $G$ on $A^G$ has sofic UPMD.
\end{prop}

The following result follows from Proposition \ref{prop:upmd}. 
\begin{corollary}\label{cor:cube_shift_pairs} Let $A$ be a path-connected compact metrizable space, $G$ a sofic group, $G \curvearrowright A^G$ the left shift action of $G$ on $A^G$,
If $x\neq y\in A^G$, then $(x,y)\in \D_ms(A^G,G)$.
\end{corollary}

Now we will show that every system with positive sofic mean dimension has sofic mean dimension pairs.

\begin{lemma}

\label{lem:subaction pairs}
Let $\act$ be an action, $\Sigma$ a sofic approximation sequence for $G$, and $Y\subset X$ a $G$-invariant subset. If $(x,y)\in \D_ms(Y,G)$ then $(x,y)\in \D_ms(X,G)$.
\end{lemma}
\begin{proof}
   Let $\alpha$ be a standard open cover of $X$, distinguishing $(x,y)\in Y\times Y\setminus \Delta_Y$ in $X$. Clearly $\alpha|_{Y}$ its restriction to $Y$, is a standard open cover of $Y$, distinguishing $(x,y)$ in $Y$. By the definition of mean dimension it holds (see proof of \cite[Proposition 2.11]{L12}) 
$$
\mdim_{\Sigma} (\alpha)\geq \mdim_{\Sigma} (\alpha|_{Y}).
$$
As $\mdim_{\Sigma} (\alpha|_{Y})>0$ this concludes the result.  
\end{proof}

\begin{prop}\label{prop:standard_pos_mdim}
Let $\act$ be an action, $\Sigma$ a sofic approximation sequence for $G$. If $\mdim_{\Sigma}(X,G)>0$ then there exists a standard open cover $\alpha$ such that $\mdim_{\Sigma} (\alpha)>0.$
\end{prop}
\begin{proof}
There exists a finite open cover $\alpha=(U_1,...,U_k)$ such that $$\mdim_{\Sigma} (\alpha)>0.$$ Without loss of generality, one may assume that for every $i\in [k]$,
\begin{equation}\label{eq:genuine cover}
U_i\setminus \bigcup_{j\in [k]\setminus \{i\}}U_j\neq \emptyset    
\end{equation}
 Given $i\in [k]$ and $x\in U_i$, let $B^i(x)$ be an open ball such that $\overline{B^i(x)}\subset U_i$. Let $(W_1,...W_p)$ be a finite subcover of $$\{B^i(x)\}_{i\in [k],x\in U_i}.$$
 \noindent
 By Equation \eqref{eq:genuine cover}, for every $i\in [k]$, there exists $j\in [p]$ such that $\overline{W_j}\subset U_i$ (equivalently $U_i^c\subset \overline{W_j}^{c}$, where the latter is an open set). Let  $$\C_i=\{\overline{W_j}^{c}:U_i^c\subset \overline{W_j}^{c}\}.$$ 
 Thus $\C_i\neq \emptyset$ for all $i\in [k]$. Moreover, if $U_i^c\neq \emptyset$, then 
 $$ U_i^c\subset V_i:=\bigcap_{Z\in \C_i} Z\neq \emptyset.$$
 \noindent
 We conclude that $\beta_i=(U_i,V_i)$ is a cover of $X$ for every $i\in[k]$.  Note that for every $j\in [p]$ there exists $i\in [k]$ so that  $\overline{W_{j}}\subset U_{i}$. Thus, every $\overline{W_j}^{c}$ appears in some $\C_i$, which implies $\bigcup _{i\in [k]}\bigcup_{Z\in \C_i} Z^c=X$. Thus:
 $$\bigcap _{i\in [k]}\bigcap_{Z\in \C_i} Z =\bigcap _{i\in [k]} V_i=\emptyset.$$

\noindent
Note that if $W\in \bigvee_{i=1}^k \beta_i$ then either $W\subset U_i$ for some $i$ or apriori $W=\bigcap _{i\in [k]} V_i$. However we showed above that this set is empty thus for all $W\in \bigvee _{i=1}^k \beta_i$, $W\subset U_i$ for some $i$  which implies  $\bigvee _{i=1}^k \beta_i\succ\alpha$. Using Lemma \ref{lem:basic} one obtains that
$$
0<\mdim_{\Sigma} (\alpha)\leq \mdim_{\Sigma} (\bigvee _{i=1}^k \beta_i)
$$
$$
\leq \sum _{i=1}^k \mdim_{\Sigma} (\beta_i).
$$
\noindent
Hence, there exists $n\in [k]$ such that $\mdim_{\Sigma} (\beta_n)>0$. If $U_n$ and $V_n$ are not dense, then $\beta_n$ is a standard open cover and the proof is complete. Assume that $U_n$ is dense. There exists $x\in U_n$ such that $\overline{B}_{\epsilon}(x)\subset V_n$. It holds that $\gamma:=(U_n\setminus \overline{B}_{\epsilon}(x),V_n)$ is also an open cover of $X$. Furthermore, since $\gamma$ is finer than $\beta_n$, we conclude that $$\mdim_{\Sigma} (\gamma)>0.$$ 
\noindent
If $V_n$ is not dense then $\gamma$ is a standard open cover and the proof is complete.  If $V_n$ is dense, then using the argument above, we find $y\in V_n$ such that $\overline{B}_{\epsilon'}(y)\subset U_n\setminus \overline{B}_{\epsilon}(x)$ and continue as above.
\end{proof}

\begin{prop}
\label{prop:mdimpair}
Let $(X,G)$ be an action, $\Sigma$ a sofic approximation sequence for $G$ and $\alpha=(U,V)$ a standard cover of $X$. If  $$\mdim_{\Sigma} (\alpha)>0,$$ then there exists $x\in U^c$ and $y\in V^c$ such that  $(x,y)\in \D_ms(X,G).$

\end{prop}
\begin{proof}

As $(U,V)$ is a standard cover of $X$, it holds $U^c\neq \emptyset$ and $V^c\neq \emptyset$. 
Fix $n\in \N$. Starting with the cover $(U,V)$, we will construct an open set $U_n$ with certain properties listed below using only the fact that $U^c\neq \emptyset$. This will allow us to repeat the construction under the same assumption on another (not necessarily standard) two-set cover in the sequel. Let $\beta=(W_1,...,W_k)$ be a finite cover of $U^c$ composed of balls of radius $\frac{1}{n}$ centered at points in $U^c$. Denote $F_i=U^c\cap \overline{W_i}$ for every $i\in [k]$. Note that for every $i$,  $F_i\subset U^c$. Thus 
$(F_i^c,V)$ is a cover of $X$. Moreover as $W_i$ is centered in $U^c$, it holds $U^c\cap W_i\neq \emptyset$ which implies $F_i\neq \emptyset$. We claim $\bigvee_{i=1}^k (F_i^c,V)\succ (U,V)$. It is enough to show $\bigcap_{i=1}^k F_i^c\subset U$, equivalently $U^c\subset \bigcup_{i=1}^k F_i=U^c\cap \bigcup_{i=1}^k \overline{W_i}$ which trivially holds.
Using Lemma \ref{lem:basic} it holds
$$
0<\mdim_{\Sigma} ((U,V))\leq \mdim_{\Sigma} (\bigvee _{i=1}^k (F_i^c,V))
$$
$$
\leq \sum _{i=1}^k \mdim_{\Sigma} ((F_i^c,V)).
$$
Define $U_n:=F_i^c$ for some $i$ with $\mdim_{\Sigma} (F_i^c,V))>0$. The key properties of $U_n$ are:
\begin{enumerate}
  \item $(U_n,V)$ is a cover of $X$,
    \item $\mdim_{\Sigma} (U_n,V)>0$,
      \item $U_n^c\neq \emptyset$,
      \item $\diam(U_n^c)\leq \frac{2}{n}$,
        \item $U_n^c\subset U^c$.
\end{enumerate}
\noindent
We now apply the above argument to the cover $(U_n,V)$ using the fact $V^c\neq \emptyset$. We may thus construct an open set $V_n$ with the following properties:
\begin{enumerate}
  \item  $(U_n,V_n)$ is a cover of $X$,
    \item $\mdim_{\Sigma} ((U_n,V_n))>0$,
      \item $V_n^c\neq \emptyset$,
      \item $\diam(V_n^c)\leq \frac{2}{n}$,
        \item $V_n^c\subset V^c$.
\end{enumerate}

 \noindent 
 Let $\{x\}=\bigcap_n U_n^c\subset U^c$ and $\{y\}=\bigcap_n V_n^c\subset V^c$. Note that as $(U,V)$ is a cover, it holds $U^c\cap V^c=\emptyset$, so $x\neq y$. Now we prove that $(x,y)\in \D_ms(X,G)$.

Let $(V',U')$ be an open cover of $X$ which distinguishes $(x,y)$. There exists $\epsilon>0$ such that $B_{\epsilon}(x)\subset U'^c$ and $B_{\epsilon}(y)\subset V'^c$. Let $m>0$ with $\frac{1}{m}<\frac{\epsilon}{2}$. It follows that $x\in U_m^c\subset B_{\epsilon}(x)$ and $y\in V_m^c\subset B_{\epsilon}(y)$. Thus $U'\subset (B_{\epsilon}(x))^c\subset U_m$ and $V'\subset (B_{\epsilon}(y))^c\subset V_m$. Conclude 
$$ (U',V')\succ (U_m,V_m).$$
\noindent
Applying  Lemma \ref{lem:basic} it holds
$$
0<\mdim_{\Sigma} ((U_m,V_m))\leq \mdim_{\Sigma}((U',V')).
$$
Thus, $(x,y)\in \D_ms(X,G)$. 

\end{proof}
\begin{corollary}\label{cor:zmd}
Let $\Sigma$ be a sofic approximation sequence for $G$. An action $(X,G)$ has positive sofic mean dimension if and only if $\D_ms(X,G)\neq \emptyset.$
\end{corollary}
\begin{proof}
   Assume $\D_ms(X,G)\neq \emptyset$. Let $(x,y)\in \D_ms(X,G)$. Let $\epsilon>0$ so that $\overline{B}_\epsilon(x)\cap \overline{B}_\epsilon(y)=\emptyset$. Define $U=X\setminus  \overline{B}_\epsilon(y)$ and $V=X\setminus \overline{B}_\epsilon(x)$. Thus $U\cup V=X$ Note $B_\epsilon(y)\subset \overline{U}^c$ and $B_\epsilon(x)\subset \overline{V}^c$. Thus $\{U,V\}$ is a standard open cover which distinguishes $(x,y)$. Conclude $ \mdim_\Sigma (X,G)\geq \mdim_\Sigma(\alpha)>0$. The opposite direction is given by Proposition \ref{prop:standard_pos_mdim} and Proposition \ref{prop:mdimpair}. 
\end{proof}

The following question is related to Question \ref{ques:local}.
\begin{question}
Given a group action, is every (sofic) mean dimension pair a (sofic) entropy pair?
\end{question}

\subsection{Completely positive mean dimension through mean dimension pairs}

Let $\pi\colon X\to Y$ be a continuous surjective function. For every $n\in \N$ we denote the induced map on the product space by $\pi_{[n]}\colon X^{[n]}\to Y^{[n]}$. If $\alpha$ is a finite open cover of $Y$, the finite open cover of $X$ consisting of preimages of elements of $\alpha$ will be denoted by $\pi^{-1}\alpha$.

The following lemma follows from \cite[Lemma 2.10]{L12}.

\begin{lemma}
[Li]
\label{lem:Li}
 Let $\Sigma=\{s_i \colon G\rightarrow \Sym(n_i) \}_{i=1}^{\infty}$ be a sofic approximation sequence for $G$, $\pi: \act \to G\curvearrowright Y$ a factor map, $d_X,d_Y$ compatible metrics on $X,Y$, $F\in \mathcal{F}(G)$ a finite subset, and $\delta>0$. Then there exists $\delta' > 0$ such that for every $n\in\N$, $s:G\to \Sym(n)$ and $p\in\Map(d_X,F,\delta',s)$ one has $\pi_{[n]}\circ p\in \Map(d_Y,F,\delta,s)$.   
\end{lemma}

\begin{lemma}
\label{lem:factor}
Let $\pi:\act\rightarrow G\curvearrowright Y$ be factor map and $(U,V)$ an open cover of $Y$. It holds that
$$
\mdim_{\Sigma} (\pi^{-1}(U,V))\leq\mdim_{\Sigma} ((U,V)).
$$
\end{lemma}
\begin{proof}


For every open cover, $\alpha$, of $Y^{[n]}$ and $p\in X^{[n]}$ it holds 
\[
\sum_{U'\in \alpha}1_{U'} (\pi_{[n]}(p))
 =\sum_{V'\in \pi_{[n]}^{-1 }\alpha}1_{V'} (p).
 \]
\noindent
Also note that if $\beta \succ (U,V)^{[n]}$, then $\pi_{[n]}^{-1}\beta \succ (\pi^{-1}(U,V))^{[n]}=\pi_{[n]}^{-1}(U,V)^{[n]}.$

Let $\delta>0$ and $i\in \N$.
 Using the previous comments and Lemma \ref{lem:Li}, there exists $\delta'>0$ so that

\begin{align*}
\min_{\beta \succ (\pi^{-1}(U,V))^{[n_i]}} \left(\max_{p \in \Map(d_X,F,\delta',s_i)} \sum_{U \in \beta} 1_{U}(p)\right) 
&\leq \min_{\beta \succ (U,V)^{[n_i]}} \left(\max_{p \in \Map(d_X,F,\delta',s_i)} \sum_{U \in \beta} 1_{U}(\pi_{[n_i]}(p))\right)  \\
&\leq \min_{\beta \succ (U, V)^{[n_i]}} \left(\max_{q \in \Map(d_Y,F,\delta,s_i)} \sum_{U \in \beta} 1_{U}(q)\right).
\end{align*}

Conclude 
\[
 \DD((\pi ^{-1}(U,V))^{[n_i]}|_{\Map(d_X, F, \delta', s_i)})\leq  \DD((U,V)^{[n_i]}|_{\Map(d_Y, F, \delta, s_i)}).
\]

Using this inequality and the formula for sofic mean dimension we obtain the result. 
 
\end{proof}
\begin{rem}\label{rem:strict_inequality}
Note that the inequality in Lemma \ref{lem:factor} may be strict. Indeed, let $\pi:\K\rightarrow \I$ be a continuous surjective map from the Cantor set to the interval. Recall that $\mdim_{\Sigma}(\K^\Z,\shift)=0 $ (\cite[Proposition 3.1]{LW}) and consider the map $\pi^{\otimes \Z}: (\K^\Z,\shift)\rightarrow (\I^\Z,\shift)$. 
\end{rem}
\begin{prop}
\label{prop:factor1}
Let $\pi:(X,G)\rightarrow (Y,G)$ be factor map. If $(x,y)\in \D_ms(X,G)$ and $\pi(x)\neq \pi(y)$ then $(\pi(x),\pi(y))\in \D_ms(Y,G)$.
\end{prop}
\begin{proof}
Let $(x,y)\in \D_ms(X,G)$ with $\pi(x)\neq \pi(y)$ and $(U,V)$ an open cover that distinguishes $(\pi(x),\pi(y))$. By continuity of $\pi$, $(\pi ^{-1}(U))^c=\pi ^{-1}(U^c)$ contains an open neighbourhood of $x$ and $(\pi ^{-1}(V))^c$ contains an open neighbourhood of $y$. Hence, $(\pi ^{-1}(U), (\pi ^{-1}(V))$ is an open cover that distinguishes $(x,y)$. Furthermore, since by Lemma \ref{lem:factor} 
$$
0<\mdim_{\Sigma} ((\pi ^{-1}(U), \pi ^{-1}(V)))\leq \mdim_{\Sigma}((U,V)),
$$
conclude that $(\pi(x),\pi(y))\in \D_ms(Y,G)$.
\end{proof}
\begin{lemma}
\label{lem:diagonal}
Let $(X,G)$ be an action, $\Sigma$ a sofic approximation sequence for $G$ and $(x,y)\in \overline {\D_ms(X,G)}$. If $x\neq y$ then $(x,y)\in \D_ms(X,G)$. 
\end{lemma}
\begin{proof}
Let $\{(x_n,y_n)\}_{n\in \N}\subset \D_ms(X,G)$ with $(x_n,y_n)\rightarrow (x,y)$, and $(U,V)$ an open cover that distinguishes $(x,y)$. There exist $n>0$ such that $(U,V)$ distinguishes $(x_n,y_n)$. Hence, $\mdim_{\Sigma} ((U,V))>0$. Thus  $(x,y)\in \D_ms(X,G)$. 
\end{proof}

\begin{rem}
 Let $\pi\colon \act \rightarrow G\curvearrowright Y$ be factor map and $(x,y)\in \D_ms(Y,G)$. It is possible that $\pi^{-1}(x)\times \pi^{-1}(y)\cap  \D_ms(X,G)= \emptyset$. 
 Indeed, consider again the map $\pi^{\otimes \Z}\colon (K^\Z,\shift)\rightarrow (\I^\Z,\shift)$ of Remark \ref{rem:strict_inequality}. 
\end{rem}

\noindent

Recall the correspondence between factors and equivalence relations in Proposition \ref{prop:factor-eq_rel}.

\begin{thm}\label{thm:eqrelation}
Let $\Sigma$ be a sofic approximation sequence for $G$,  $(X,G)$  an action with $\mdim_\Sigma (X,G)\geq 0$ and $\R_u$ the equivalence relation which corresponds to its universal zero mean dimension factor $(\X_u, G)$. The following holds:
\begin{enumerate}
    \item 
    
    The smallest closed $G$-invariant equivalence relation that contains $\D_ms(X,G)$ is contained in $\R_u$. 
    \item
    The action $\act$ has sofic CPMD if the smallest closed $G$-invariant equivalence relation that contains $\D_ms(X,G)$ is $X\times X$. 
\end{enumerate}

\end{thm}
\begin{proof}
We start by proving (1). Let $Q$ be the smallest closed $G$-invariant equivalence relation that contains $\D_ms(X,G)$ and let $\R_u$ be the equivalence relation which corresponds to $(\X_u,G)$. Note it is enough to prove $\D_ms(X,G)\subset \R_u$ in order to prove $Q\subset \R_u$. Let $(x,y)\in \D_ms(X,G)$ and assume for a contradiction that $(x,y)\notin \R_u$. By Lemma \ref{lem:univ_sep}, there is a factor map $\pi:(X,G)\rightarrow (Y,G)$ with $\mdim_{\Sigma}(Y,G)\leq 0$ so that $\pi(x)\neq \pi (y)$. By Proposition \ref{prop:factor1}, $(\pi(x),\pi(y))\in \D_ms(Y,G)$. By Corollary \ref{cor:zmd}, $\mdim_{\Sigma}(Y,G)>0$. Contradiction. Conclude $Q\subset \R_u$. 

We now claim (2) follows from (1). Indeed, if the smallest closed $G$-invariant equivalence relation that contains $\D_ms(X,G)$ is $X\times X$, then by (1), $\R_u=X\times X$. By the definition of $\R_u$ there cannot be a non-trivial factor with non-positive sofic mean dimension.  
\end{proof}

A direct consequence of Theorem \ref{thm:eqrelation} is the following.

\begin{corollary}\label{cor:upmd implies cpmd}
Every non-trivial action\footnote{That is, it has more than one element.} that has sofic UPMD has sofic CPMD. 
\end{corollary}
\begin{proof}
    Let $\act$ be an action with sofic UPMD. By Proposition \ref{prop:upmd},  $\D_ms(X,G)= (X\times X)\setminus \Delta_X$. As $|X|\geq 2$, $X$ has standard open covers and thus as it has sofic UPMD it holds $\mdim_\Sigma (X,G)>0$. Thus we may invoke Theorem \ref{thm:eqrelation} to conclude $(X\times X)\setminus \Delta_X\subset \R_u$ which implies $\R_u=X\times X$. Invoking Corollary \ref{cor:cpmd_iff_umdz_trivial} completes the proof.
\end{proof}

We now recall the important result of Lindenstrauss and Weiss that the shift $\Z\curvearrowright [0,1]^\Z$ has completely positive mean dimension (\cite[Theorem 3.6]{LW}, which was generalized by Li for actions of sofic groups (\cite[Theorem 7.5]{L12}).
\begin{thm}\label{thm:Bernoulli_CPMD}
[Lindenstrauss and Weiss; Li]
Let $A$ be a path-connected compact metrizable space, and $G$ a sofic group. The left shift action of $G$ on $A^G$ has sofic CPMD.
\footnote{There also exist minimal systems with completely positive mean dimension (see \cite[Page 12]{LW} and \cite[Section 5]{L95}).} 
\end{thm}
\begin{proof}
    By Corollary \ref{prop:AG_UPMD} $(A^G,G)$ has sofic UPMD. By Corollary \ref{cor:upmd implies cpmd} it has  sofic CPMD.
\end{proof}

\begin{question}
 Does there exist an action $\act$ with sofic CPMD such that the smallest closed $G$-invariant equivalence relation that contains $D_{md}(X,G)$ is not $X\times X$?   
\end{question}

\section{The mean dimension map is Borel}\label{sec:mdim_Borel}

Given a metrizable space, $(Y,d)$, denote by $K(Y)$ the space of all non-empty compact subsets of $Y$.  We will use the \textbf{Hausdorff metric} as follows on $K(Y)$. For $A,B\in K(Y)$, define
$$
d_H(A,B)=\inf \{\delta>0: A\subset B_{\delta}\text{ and } B\subset A_{\delta}\},
$$
where $C_{\delta}=\{x\in Y:d(x,C)<\delta\}$.
If $Y$ is compact, then so is $K(Y)$. For more information, see \cite[Chapter 4.F]{K95}. We equip $K(Y)$ with the topology given by the Hausdorff metric.

\begin{rem}
\label{rem:sofic}
For the rest of the paper we fix a compact metrizable space $A$, a sofic group $G$ and  a sofic approximation sequence $\Sigma$ for $G$. Also, we will use $G\curvearrowright A^G$ to denote the left shift action and $d$ to denote a compatible metric of the product topology $A^G$.
\end{rem}
Define 
$$
\S(A)=\{X\subset A^{G}:X\text{ is closed non-empty, and } \forall g\in G\,\,\,gX= X\}.
$$
Note $\S(A)$ is a closed subset of $K (A^{G})$. Thus $\S(A)$ equipped with the induced topology, is a compact metrizable space, and hence  Polish.

Let $Y$ be a metrizable space. Recall that a function $f\colon Y\rightarrow \overline{\RR}$ is called \textbf{lower semi-continuous} if for all $y\in Y$,  $
 \liminf_{y'\rightarrow y} f(y')\geq   f(y).
 $

\begin{lemma}\label{lem:semi-cont}
Let $F\in \mathcal{F}(G)$, $n\in \N$, $\delta>0$, $\beta$ a finite open cover of $(A^G)^{[n]}$ and $\Sigma\ni s\colon G\to	\Sym(n) $. The function
$$
f_{F,\delta,s,\beta}\colon \S(A)\rightarrow \N\cup \{0,-\infty\},
$$
defined by $f_{F,\delta,s,\beta}(X)= \ord(\beta|_{\Map'(d, F, \delta, s,X)})$ is lower semi-continuous and hence Borel. 
\end{lemma}

\begin{proof}

Fix $X\in \S(A)$. We will show $f_{F,\delta,s,\beta}$ is lower semi-continuous at $X$ by showing a slightly stronger property. Namely, we will show that there exists an open neighborhood $\mathcal{W}$ of $X$ such that for every $Y\in \mathcal{W}$, it holds $f_{F,\delta,s,\beta}(Y)\geq f_{F,\delta,s,\beta}(X)$.
If $\Map'(d, F, \delta, s,X)=\emptyset$, equivalently $f_{F,\delta,s,\beta}(X)=-\infty$, then there is nothing to prove. Assume $\Map'(d, F, \delta, s,X)\neq\emptyset$ and fix
 $\phi_0\in \Map'(d, F, \delta, s,X)$ such that $$\ord(\beta|_{\Map'(d, F, \delta, s,X)})=\sum_{U\in \beta}1_{U}(\phi_0)-1.$$
 For every $g\in F$ it holds, $$n\delta_g^2:=\sum_{v=1}^{n}d^{2}(\phi_0(s_g(v)),g\phi_0(v))< n\delta^2.$$
 Denote $D=\diam(A^G)$. Choose $\epsilon>0$ such that for all $g\in F$
 \begin{align}
 \label{f1}
2\epsilon^2 +6D\epsilon + \delta_g^2< \delta^2. 
 \end{align}
  Choose $0<\epsilon'<\epsilon$ such that  for all $x,y\in X$ and all $g\in F$ (a finite set) with $d(x,y)\leq \epsilon'$ implies
  \begin{align}
 \label{f2}
d(gx,gy)\leq \epsilon.
\end{align}
Let $0<\rho<\epsilon'$ be small enough so that the open ball of radius  $\rho$ with respect to $d_{\infty}$, the maximum
 metric on $X^{[n]}$, centered at $\phi_0$ is contained in the open set $\bigcap_{\phi_0\in U\in \beta}U$. We will denote this ball by $V\subset (A^G)^{[n]}$. Note that for every $\phi\in V$
\begin{align}
 \label{f3}
  \sum_{U\in \beta}1_{U}(\phi)\geq \sum_{U\in \beta}1_{U}(\phi_0)=\ord(\beta|_{\Map'(d, F, \delta, s,X)})+1.
 \end{align}
Let $\mathcal{W}:=\{X'\in \S(A)|\, d_H(X,X')<\rho\}$ be an open ball in $\S(A)$. Fix $X'\in \mathcal{W}$. By the definition of the Hausdorff distance one may choose a function   
$\phi_1:[n]\to X'$ ($X'$, not $X$) having the property that for all $v\in [n]$,
\begin{align}
\label{f4}
  d(\phi_0(v),\phi_1(v))\leq d_H(X,X')<\rho<\epsilon'.
 \end{align}
Thus $\phi_1\in V$. Fix $g\in F$. We now calculate using Equations  \eqref{f2} and \eqref{f4}:
 \begin{align*}
\sum_{v=1}^{n}d^{2}(\phi_1(s_g(v)),g\phi_1(v))
&\leq 
\sum_{v=1}^{n}
\big(d(\phi_1(s_g(v)),\phi_0(s_g(v)))+ d(\phi_0(s_g(v)),g\phi_0(v))+d(g\phi_0(v),g\phi_1(v))\big)^2\\
&:=\sum_{v=1}^{n} (a_v+b_v+c_v)^2
\end{align*}
Using this notation notice that for all $v$, $a_v< \epsilon'< \epsilon$ and $c_v\leq \epsilon$. In addition $\sum_{v=1}^{n} b_v^2=n\delta_g^2$.  Thus we may calculate using Equation \eqref{f1}:
\begin{align*}
\sum_{v=1}^{n}d^{2}(\phi_1(s_g(v)),g\phi_1(v))
&\leq \sum_{v=1}^{n} (a_v^2+b_v^2+c_v^2 +2 a_v b_v+2 a_v c_v+2 b_v c_v)
\\
&< n (\epsilon^2+ \delta_g^2+ \epsilon^2 +6D\epsilon)< n \delta^2.
\end{align*}
\noindent
Conclude $\phi_1\in \Map'(d, F, \delta, s,X')$. 
Thus by  Equation \eqref{f3}, $\ord(\beta |_{\Map'(d, F, \delta, s,X')})\geq -1+\sum_{U\in \beta}1_{U}(\phi_1)\geq \ord(\beta |_{\Map'(d, F, \delta, s,X)})$ which was what we seeked to show. 
    
\end{proof}

\begin{deff} \label{def:U}
Let $Y$ be a compact metrizable space. For every $n\in\N$ we fix a finite open cover of $Y$ of balls of diameter $1/n$ and denote it by $\alpha_n(Y)$. Let $\mathcal{U}_n(Y)$ be the family of open covers generated by unions of elements of $\alpha_n(Y)$, and $\mathcal{U}(Y)=\bigcup_{n\in \N} \mathcal{U}_n(Y)$.
If $Y=A^G$, then we drop $(Y)$ from the notation, i.e., we write $\alpha_n$ and $\mathcal{U}$.

\end{deff}

\begin{lemma}
\label{lem:U_construction}
   Let $Y$ be a compact metrizable space and  $\gamma=\{U_1,\ldots, U_m\}$  a finite open cover of $Y$. There exists a finite open cover $\beta=\{V_1,\ldots V_m\}\in \mathcal{U}(Y)$ so that for all $i$, $V_i\subset U_i$. 

\end{lemma}

\begin{proof}
 By Lebesgue's number lemma, there exists $n\in \N$ such that $\alpha_n\succ \gamma$.  Define
 $$V_i=\cup \{W\in \alpha_n: W\subset U_i\}.$$ 

 Clearly $\beta=\{V_1,\ldots, V_m\}$ has the desired properties. 

\end{proof}

\begin{deff}\label{def:gamma_U}
    Let $\gamma$ be a finite open cover of $A^G$. In the sequel, we denote the open cover $\beta$ from the previous lemma by $\gamma|_{\mathcal{U}}.$\footnote {In order to have a unique $\beta$, one can take in the proof of Lemma \ref{lem:U_construction} the minimal $n\in \N$ such that $\alpha_n\succ \gamma$.}
\end{deff}

Recall Definition \ref{def:D_subset}.
\begin{lemma}
\label{lem:dim}
   Let $Y$ be a compact metrizable space and $\alpha$ be a finite open cover of $Y$. Suppose $\emptyset \neq Z\subset Y$ ($Z$ not necessarily closed). It holds that 
    $$
\DD(\alpha|_{Z})=\min_{\mathcal{U}(Y)\ni\beta \succ \alpha}\ord(\beta|_{Z}).
$$
\end{lemma}

\begin{proof}
Let $\beta'\succ \alpha$ and write $\beta'=\{U_1,\ldots, U_m\}$. Denote $\beta=\beta'_{\mathcal{U}}\in \mathcal{U}(Y)$. As $\beta=\{V_1,\ldots, V_m\}$ for some $V_i\subset U_i$ for all $i=1,\ldots, m$ it holds $\ord(\beta|_{Z})\leq \ord(\beta'|_{Z})$. This concludes the proof. 
\end{proof}

Recall Remark \ref{rem:notation}.

\begin{lemma}
\label{cor:coverborel}
    
    Let $\alpha$ a finite open cover of $A^G$. The function\footnote{The ranges of various mean dimension functions is discussed in Subsection \ref{subsec:sofic mdim}.} 
   $g_\alpha\colon \S(A)\rightarrow \mathbb{R}\cup \{-\infty\}$ defined by $g_\alpha(X)=\mdim_{\Sigma}(\alpha|_X)$ is a Borel function.
    In particular $$
    \{X\in \S(A):\mdim_{\Sigma}(\alpha|_X)>0\}
    $$

    is a Borel subset of $\S(A)$. 
\end{lemma}
\begin{proof}
Note that part of the statement is that the range of  $g_\alpha$ is $\mathbb{R}\cup \{-\infty\}$ (in particular it does not include $\infty$).
By Lemma \ref{lem:semi-cont}, the function
$
f_{F,\delta,s,\beta}\colon \S(A)\rightarrow \N\cup \{0,-\infty\}
$
defined by $f_{F,\delta,s,\beta}(X)= \ord(\beta|_{\Map'(d, F, \delta, s,X)})$ is  Borel. 
\noindent
By Lemma \ref{lem:D'}, equation \eqref{eq:D'} and  Lemma \ref{lem:dim}, we obtain
\begin{align*}
    g_\alpha(X)=\mdim_\Sigma(\alpha|_X) &= 
\lim_{(F, \delta)\to (\infty,0)} 
\varlimsup_{i\to\infty} \frac{\DD'(\alpha|_X, d, F, \delta, s_i)}{n_i}\\ &= \lim_{(F, \delta)\to (\infty,0)} \varlimsup_{i\to\infty} \frac{\DD((\alpha|_X)^{[n]}|_{\Map'(d, F, \delta, s_i,X)})}{n_i}\\&= \lim_{(F, \delta)\to (\infty,0)} \varlimsup_{i\to\infty} \frac{\min_{\mathcal{U}(X^{[n]})\ni\beta \succ (\alpha|_X)^{[n]}}f_{F,\delta,s_i,\beta}(X)}{n_i}.
\end{align*}
\noindent
We thus have three successive operations $\min_{\mathcal{U}(X^{[n]})}$, $\varlimsup_{i\to\infty}$, $\lim_{(F, \delta)\to (\infty,0)}$ on countable collections of functions (here we use that $\mathcal{U}$ is countable). The operations of (upper) limits and minimum over countable collections of  Borel functions result with Borel functions (this is proven similarly to \cite[Theorem 11.6]{K95}). It follows that $g_\alpha$ is a Borel function, and  $
    \{X\in \S(A):\mdim_{\Sigma}(\alpha|_X)>0\}
    $ a Borel subset of $\S(A)$.

\end{proof}

\begin{thm}
\label{thm:dimensionmapBorel}
        The function $f\colon \S(A)\rightarrow \mathbb{R}\cup \{\pm\infty\}$ defined by $f(X)=\mdim_{\Sigma}(X,G)$ is Borel. 
\end{thm}
\begin{proof}
  Note that $$\mdim_\Sigma (X,G) = \sup_{\alpha\in \mathcal{U}} g_\alpha(X).$$   
Since $\mathcal{U}$ is countable, the result follows from Lemma \ref{cor:coverborel}. 
\end{proof}
Define 
\[
\S_{+}(A)=
\{
X\in \S(A): \mdim_{\Sigma}(X,\shift)>0 \}.
\]
\begin{corollary}
\label{cor:Sborel}
The set $\S_{+}(A)$ is a Borel subset of $\S(A)$. 
\end{corollary}

\section{Descriptive complexity of UPMD and CPMD}\label{sec:descriptive_complexity}
In this section we study the descriptive complexity of the following sets, 
\[
\S_{c+}(A)=
\{
X\in \S(A): (X,\shift)\textup{ has sofic CPMD}
\}
\]
\[
\S_{u+}(A)=
\{
X\in \S(A): (X,\shift)\textup{ has sofic UPMD}
\}.
\]
Recall Remark \ref{rem:sofic}, Definition \ref{def:standard} of standard open covers and Definition \ref{def:U} of $\mathcal{U}$.
Let $Y\in \S (A)$. Define
$$
\mathcal{U}_{St}(Y)=\{\alpha\in \mathcal{U}: \alpha=\{U,V\},\text{ } Y\cap \overline{U}\neq Y \text{ and }  Y\cap \overline{V}\neq Y\}.
$$
Note that $\mathcal{U}_{St}(A^G)$ is exactly the collection of standard open covers of $A^G$ belonging to $\mathcal{U}$, however $\mathcal{U}_{St}(Y)$ for $Y\subsetneq A^G$ may be strictly contained in the collection of open covers belonging to $\mathcal{U}$ which are  standard with respect to $Y$ when restricted to $Y$. \\
\noindent
We denote $\mathcal{U}_{St}(A^G)$
simply by $\mathcal{U}_{St}$.
\begin{lemma}
\label{lem:standard}
    Let {$X\in \S (A)$} and $\alpha$ a standard open cover of $X$. Then there exists $\beta\in \mathcal{U}_{St}(X)$, such that $\beta|_{X}=\alpha$.
    
\end{lemma}

\begin{proof}
Let $(x,y)\in X\times X$ a pair distinguished by $\alpha$. There exist open subsets $V,W\subset A^G$ so that $\alpha=\{W\cap X, V\cap X\}$, $x\in W$ and $y\in V$.
Let 
$$0<2\epsilon=\min \{d(x,\overline{V\cap X}),d(y,\overline{W\cap X})\}.  $$
\noindent
Define $W'=W\setminus \overline{B}_\epsilon(y)$ and $V'=V\setminus \overline{B}_\epsilon(x)$ (balls in $A^G$). Since $W'\cap X=W\cap X$ and $V'\cap X=V\cap X$, it holds that
$\{W',V'\}$ is a cover of $X$.
\noindent
 Every metrizable space is normal, that is, every two disjoint closed sets have disjoint open neighborhoods \cite[Example 15.3]{willard2012general}. This implies that there exists an open subset $U$ of $A^G$ so that

\begin{align}
\label{formula}
X\subset U\subset \overline{U}\subset W'\cup V'.
\end{align}
\noindent
Define 
$$\beta=\{W', V'\cup (X^c\setminus \overline{U}\}.$$ By \eqref{formula}, $\beta$ is a cover of $A^G$ and $\beta|_{X}=\{W'\cap X,V'\cap X\}=\alpha$.
Note $y\notin \overline{W'}$ and $x\notin \overline{V'}$. By \eqref{formula}, $x\notin X^c\setminus U=U^c$. As $\overline{X^c\setminus \overline{U}}\subset X^c\setminus U$, it holds  $x\notin \overline{V'\cup (X^c\setminus \overline{U})}$. Conclude that $\beta\in\mathcal{U}_{St}(X)$.

\end{proof}

\begin{lemma}
\label{lem:open}
    Let $\alpha\in \mathcal{U}_{St}$. Then $\{X\in \S(A):\,\alpha\in \mathcal{U}_{St}(X) \}$ is open. 
\end{lemma}
\begin{proof}
 Let $\alpha=\{U,V\}\in \mathcal{U}_{St}$ and $X\in \S(A)$ so that $\alpha\in \mathcal{U}_{St}(X)$. By definition $X\cap \overline{U}\neq X$ and $X\cap \overline{V}\neq X$. Thus, one may choose $\epsilon>0$ and $x,y\in X$ such that $B_{\epsilon}(x)\cap \overline{V}=\emptyset$ and $B_{\epsilon}(y)\cap \overline{U}=\emptyset$ (balls in $A^G$). Hence, if $X'\in \S(A)$ such that $d_H(X,X')<\epsilon$, then $X'\cap \overline{U}\neq X'$ and $X'\cap \overline{V}\neq X'$. Since the topology in $\S(A)$ is generated by the Hausdorff metric, this concludes the proof.
   
\end{proof}

\begin{thm}
 The set $\S_{u+}(A)$ is a Borel subset of $\S(A)$. 
\end{thm}

\begin{proof}
It is enough to show $\S_{u+}(A)^c$ is Borel. We claim $X\in \S_{u+}(A)^c$ if and only if there exists $\gamma\in \mathcal{U}_{St}(X)$ such that $\mdim_\Sigma (X,\gamma|_X) \leq 0$. Indeed, if $X\in \S_{u+}(A)^c$, then by definition, there exists a standard open cover $\beta$ of $X$, such that $\mdim_\Sigma (X,\beta) \leq 0$. 
 By Lemma \ref{lem:standard}, there exists $\gamma\in \mathcal{U}_{St}(X)$, such that $\gamma|_X=\beta$. Thus $$\mdim_{\Sigma}(\gamma|_X)=\mdim_\Sigma (X,\beta) \leq 0.$$ Conversely, if $\gamma\in \mathcal{U}_{St}(X)$ such that $\mdim_{\Sigma}(\gamma|_X)\leq 0$, then as $\gamma|_X$ is a standard open cover of $X$ it implies $X\in \S_{u+}(A)^c$.\\
Thus
\begin{align*}
\S_{u+}(A)^c &= \bigcup_{\alpha\in \mathcal{U}_{St}}\big(\{X\in \S(A):\mdim_{\Sigma}(\alpha|_X)=0\} \cap \{X\in \S(A) :\,\alpha\in \mathcal{U}_{St}(X)\}\big).
\end{align*}
\noindent
Finally, we use Lemma \ref{cor:coverborel}, Lemma \ref{lem:open}, and the fact that $\mathcal{U}_{St}$ is countable to complete the proof. 

\end{proof}

We say that a subset of a Polish space is \textbf{analytic} if it is the continuous image of a Borel subset of a Polish space and \textbf{coanalytic} (or, equivalently, $\Pi_1^1$) if it is the complement of an analytic set.
All Borel subsets of a Polish space are both analytic and coanalytic. Moreover, if a set is both analytic and coanalytic, then it is Borel \cite[Corollary 26.2]{K95}. However, in every uncountable Polish space there are analytic, and hence coanalytic, sets which are not Borel. Loosely speaking, if a set is analytic or coanalytic but not Borel, it means that it cannot be described with countable information.
 
We want to show that $\SC(A)$ is coanalytic. For this we have to define a model for factor systems. 

Let $Y$ and $Y'$ be compact metrizable spaces. We set $$C(Y,Y')=\{f\colon Y\rightarrow Y': f\text{ is continuous}\}.$$ We equip $C(Y,Y')$ with the topology generated by the sup metric. 

Let $\H=[0,1]^\N$ be the Hilbert cube. We will work with the following Polish space $$\mathcal{P}=K(A^G)\times K(\H^G)\times C(A^G,\H^G)$$ equipped with the product topology, where $K(A^G),K(\H^G)$ are equipped with the Hausdorff metric. 

\begin{rem}
\label{rem:hilbert}
    The advantage of working with the Hilbert cube is that it is universal, that is, every compact metrizable space embeds into the Hilbert cube \cite[Example 22.4]{willard2012general}. A consequence of this is that every TDS is conjugate to an element in $\mathcal{S}(\H)$.   
\end{rem} 

Define the factor systems as follows
\[
\begin{array}
[c]{cc}
\mathcal{F}=\{(X,Y,f)\in \mathcal{P}: X\in \S(A),Y\in\S(\H)\text{, } 
\\ 
f\colon A^G \rightarrow \H^G \text{ is continuous and G-equivariant on $X$, and }f(X)=Y \}.
\end{array}
\]

The following result is standard (e.g. see \cite[Page 140]{taylor1985general}). 

\begin{lemma}
[Moore-Osgood]
\label{lem:MO}

    Let $(Z,\tau)$ be a metric space and $(Y,\rho)$ a complete metric space. Let $x,x_1,x_2,\ldots \in Z$ with $x=\lim_{i\rightarrow\infty} x_i$ (that is $\lim_{i\rightarrow\infty} \tau (x,x_i)=0$). 
    If a sequence of functions $f_n\colon Z\to Y$ converges uniformly to $f\colon Z\to Y$, that is for every $\epsilon>0$, there exists $N\in \N$, so that for all $n\geq N$ $\rho (f_n(z),f(z))<\epsilon$ for all $z\in X$ and $f_n(x)=\lim_{i\rightarrow\infty}f_n(x_i)$ for all $n\in \N$, then 
    \[
\lim_{i\rightarrow\infty}f(x_i)=\lim_{n\to \infty}f_n(x).
    \]
\end{lemma}

\begin{lemma}
   The set $\mathcal{F}$ is a closed subset of $\mathcal{P}$.
\end{lemma}
\begin{proof}
    Let $\{(X_n,Y_n,f_n)\}_{n\in\N}\subset \mathcal{F}$ be a sequence such that $(X_n,Y_n,f_n)\rightarrow (X,Y,f)$ in $\mathcal{P}$. Since $\S(A)$ and $\S(\H)$ are closed subsets, it holds that $X\in \S(A)$ and $Y\in \S(\H)$. Let $x\in X$. There exists $x_n\in X_n$ such that $x_n\rightarrow x$. We  have that $gf_n(x_n)=f_n(gx_n)$ for all $g\in G$ and all $n\in \N$. Using the fact that every compatible metric for a compact metrizable space is complete and Lemma \ref{lem:MO} with $Z=A^G$ and $Y=H^G$ we conclude that $gf(x)=f(gx)$ for all $g\in G$. 
    
    Let $y\in Y$. There exists $x_n\in X_n$ and $x\in X$ so that $x_n\rightarrow x$ and $f_n(x_n)\rightarrow y$. Lemma \ref{lem:MO} implies that $f(x)=y$. Conclude $f(X)=Y$.
    
\end{proof}

Now we define the non-trivial zero mean dimension factor systems as follows
$$
\mathcal{F}_0(A)=\{(X,Y,f)\in \mathcal{F}: \mdim_{\Sigma}(Y)=0\text{ and }|Y|>1\}.
$$

\begin{lemma}
The set $\mathcal{F}_0(A)$ is a Borel subset of $\mathcal{P}$.
\end{lemma}
\begin{proof}
   The subset $\{(X,Y,f)\in \mathcal{P}: |Y|>1\}$ is open, and the subset $$\{(X,Y,f)\in \mathcal{P}: \mdim_{\Sigma}(Y)=0\}$$ is Borel (Corollary \ref{cor:Sborel}).
\end{proof}

\begin{thm}
    $\SC(A)$ is a coanalytic subset of $\S(A)$. 
\end{thm}
\begin{proof}
Let $X\in \S(A)\setminus \SC(A)$. Using Remark \ref{rem:hilbert}, there exists $Y\in \S(\H)$, with $\mdim_{\Sigma}(Y)=0$ and $|Y|>1$, and a continuous surjective $G$-equivariant function $f' \colon X\rightarrow Y$. Using Tietze's Extension Theorem on each coordinate of $\H$ one can show there exists a continuous function $f\colon A^G\rightarrow \H^G$ such that $f|_X=f'$. This implies that $(X,Y,f)\in \mathcal{F}_0(A)$. 

Thus, $\S(A)\setminus \SC(A)$ is the projection to the first coordinate of $\mathcal{F}_0(A)$ and hence it is an analytic subset. We conclude that $\SC(A)$ is a coanalytic subset of $\S(A)$. 
\end{proof}

If $A$ is a finite set then $\SC(A)$ is empty and hence Borel. Nonetheless, in the next section we will see that if $A$ is the unit interval then $\SC(A)$ is not Borel.

In this section we will construct an extensive family of shift spaces with completely positive sofic mean dimension. More precisely, we construct a function that assigns to each countable compact subset of $[0,1]$ a shift space on $[0,1]$ with sofic CPMD. This assignment enables us to characterize the descriptive complexity of this family.

 A standard way to prove that a coanalytic set is not Borel is to reduce it to a known set which is not Borel. More specifically, if $\mathcal{B}$ is a known non-Borel subset of some Polish space $\mathcal{Y}$, $\mathcal{A} \subseteq  \mathcal{X}$ and $f\colon \mathcal{Y}\rightarrow \mathcal{X}$ is a Borel function such that $f^{-1}( \mathcal{A}) = \mathcal{B}$, then $\mathcal{A}$ is not Borel. In this case, we say that $\mathcal{B}$ is \textbf{Borel reducible} to $\mathcal{A}$. This inspires the following definition. 
\begin{deff}
A coanalytic subset $\mathcal{A}$  of a Polish space $\mathcal{X}$ is \textbf{complete coanalytic} (or $\Pi_1^1$-complete) if for every coanalytic set $\mathcal{B}$ of a Polish space $\mathcal{Y}$, there exists a Borel function $f:\mathcal{Y}\rightarrow \mathcal{X}$ such that $f^{-1}( \mathcal{A}) = \mathcal{B}$. 
\end{deff}
In some sense, the complete coanalytic sets are as complicated as coanalytic sets can be. The following proposition simply follows from the definition and the fact that the composition of Borel functions is Borel. 

\begin{prop}
\label{prop:coanalyticbasic} Let $\mathcal{A}$ be a coanalytic subset of a Polish space $\mathcal{X}$ and $\mathcal{B}$ be a complete coanalytic subset of a Polish space $\mathcal{Y}$. 
If there exists a Borel function $f:\mathcal{Y}\rightarrow \mathcal{X}$ such that  $f^{-1}( \mathcal{A}) = \mathcal{B}$, i.e. $\mathcal{B}$ is Borel reducible to $\mathcal{A}$, then $\mathcal{A}$ is also complete coanalytic. 
\end{prop}

Denote the unit interval by $\I=[0,1]$.
{A classic example of a coanalytic complete set is the family of countable compact subsets of $\I$. For a proof of the following result see \cite[Theorem 27.5]{K95}. 
\begin{thm}
[Hurewicz] 
\label{hurewicz}
 The collection of countable compact subsets of $\I$, is a  complete coanalytic subset of $K(\I)$. \end{thm}
 }

\begin{prop}\label{cor:cpmd iff A countable}
  There exists a continuous map   $\psi\colon K(\I)\rightarrow \S(\I)$ such that $\psi(B)$ has sofic CPMD if and only if $B$ is countable.
\end{prop}

\begin{proof}

For each $B \in K(\I)$, we let $C(B)$ be the collection of all intervals contiguous to $B$, i.e, the collection of maximal connected components of $\I\setminus B$.

For $b\in B$ let $b^G$ be the point $x\in \I^{G}$ such that $x_g=b$ for all $g\in G$.

For $J\in C(B)$ we define
\[
X_J=\{x\in \I^{G}: x_g\in \overline{J} \text{ }\forall g\in G
\},
\]

and $\psi:K(\I)\rightarrow \S(\I)$ as
\[
\psi(B)=
\bigcup_{b\in B}\{b^G\}
 \cup
\bigcup_{J\in C(B)} X_J.
\]

It is not hard to see that $\psi(B)\in \S(\I)$.

We will now show that $\psi$ is continuous. We will denote by $d$ both a fixed compatible metric on $\I$ as well as on $\I^G$. Let $\epsilon>0$ and $\delta>0$ so that if for every $g\in G$ $d(w_g,z_g)<\delta$, $w_g,z_g\in I$, then $d((w_g)_{g\in G},(z_g)_{g\in G})<\epsilon$. Now fix $A,B\in K(\I)$ so that $d_H(A,B)<\frac{\delta}{4}$. Let $x\in \psi(A)$. We will show there exists $y\in \psi(B)$ so that $d(x,y)<\epsilon$. As one can exchange the roles of $A$ and $B$, this will imply $d_H(\psi(A),\psi(B))<\epsilon$. First assume $x=a^G$, $a\in A$. As $d_H(A,B)<\delta$, one may  choose $b\in B$, so that $d(a,b)<\delta$. Denote $y:=b^G\in \psi(B)$ and conclude $d(x,y)<\epsilon$ as desired. 
Now assume $x=(x_g)_{g\in G}\in \psi(A)\setminus \bigcup_{a\in A}\{a^G\}$. Thus there exists $J\in C(A)$ so that $x\in X_J$. Note $J=(c,d)$ for some $0<c<d<1$,  or $J=[0,d)$ for some $0<d<1$ or $J=(c,1]$ for some $0<c<1$. In those cases denote respectively $t_J=c$,  $t_J=d$ and  $t_J=c$. Clearly $t_J\in A$.   If $\leb(J)<\frac{3\delta}{4}$,  choose $b\in B$, so that $d(t_J,b)<\frac{\delta}{4}$ and notice that for every $g\in G$, $d(x_g,b)<\delta$. This implies for  
 $y:=b^G\in \psi(B)$, $d(x,y)<\epsilon$ as desired. If $\leb(J)\geq\frac{3\delta}{4}$ then, as $d_H(A,B)<\frac{\delta}{4}$, there exists $J'\in C(B)$ so that defining 
 $$\overline{B}_{-\frac{\delta}{4}}(J):=\{t\in J|\, \forall s\in \I\setminus J,\,\, |s-t|\geq\frac{\delta}{4}\},$$
it holds $\emptyset\neq \overline{B}_{-\frac{\delta}{4}}(J)\subset J'$. Thus one may choose for every $g\in G$, $y_g\in J'$, so that $d(x_g,y_g)<\delta$. Let $y:=(y_g)_{g\in G}$. Conclude  $d(x,y)<\epsilon$ as desired.

Now we will show that if $B$ is countable, then every non-trivial factor of $\psi(B)$ has positive mean dimension. Let $$\pi:G\curvearrowright \psi(B)\rightarrow G\curvearrowright Z.$$
If $\pi|_{X_J}$ is non-trivial for some $J \in C(B)$, then by Theorem \ref{thm:Bernoulli_CPMD}, $G\curvearrowright \pi(X_J)$  has positive mean dimension. As $G\curvearrowright \pi(X_J)$ is a subaction of $G\curvearrowright Z$, by \cite[Proposition 2.11]{L12} $G\curvearrowright Z$ has positive mean dimension.  Thus, one may assume that for all $J \in C(B)$, there exists $y_{J}\in Z$ such that $\pi(X_J)=y_{J}\in Z$. In particular $\pi(x^G)=y_{J}\in Z$ for all $x\in \overline{J}$.
Thus $\pi:G\curvearrowright \psi(B)\rightarrow G\curvearrowright Z$ induces a continuous map $f:\I\rightarrow Z$ with $f(a)=\pi(a^{G})$ for all $a\in B$ and $f(b)=y_J$ for all $b\in \overline{J}$. Note that $f$ has a countable image. This implies that there exists $p\in Z$ so that $f(\I)=\{p\}$ and as a consequence $\pi(\psi(B))=\{p\}$ as desired. Indeed by Urysohn's lemma a countable connected metric space is a singleton. Alternatively embedding $Z$ into the Hilbert cube, $\H$, we may assume, without loss of generality, that $f$ is a map into $\H$.  As $f$ is continuous, $f$ must be constant by the intermediate value theorem used on the interval factors of the Hilbert cube.

Now assume that $B$ is not countable. By the Cantor-Bendixson theorem, \cite[Theorem 6.4]{K95}, $B$ contains a set $K$ which is homeomorphic to the Cantor (ternary) set $\mathrm{C}$. Using the standard map $\pi\colon \{0,1\}^{\N}\rightarrow \mathrm{C}$, one equips $\mathrm{C}$ with the measure $\mu=\pi_*(\ber_{\frac{1}{2},\frac{1}{2}})$, where $\ber_{\frac{1}{2},\frac{1}{2}}$ is the $(\frac{1}{2},\frac{1}{2})$-Bernoulli measure. Note that $\mu$ is an atomless probability measure. Let $\mu_K$ be the corresponding measure (through the homeomorphism) on $K$. Let $f\colon \I\rightarrow \I$ be the function given by $f(t)=\int_0^t d\mu_K$. This function corresponds to the classical \textit{Cantor (or Devil’s staircase) function} associated to $K$. It is continuous, non-constant and for every open interval $Q$ in the complement of $B$, $f|_{Q}=\const$.  We now define a non-trivial factor, $\pi:G\curvearrowright \psi(B) \rightarrow G\curvearrowright \I$ where $\pi((x_g)_{g\in G})=f(x_e)$ and $G\curvearrowright \I$ acts as the identity. By Proposition \ref{thm:trivial_mdim_0}, $\mdim_\Sigma (G\curvearrowright \I)=0$.

\end{proof}

\begin{thm}
\label{thm:cc}
The set $\SC(\I)$ is a complete coanalytic subset of $\S(\I)$. 
\end{thm}

\begin{proof}
    Use Theorem~\ref{hurewicz}, Proposition~\ref{prop:coanalyticbasic} and Proposition~\ref{cor:cpmd iff A countable}. 
\end{proof}

\begin{thm}
\label{thm:ccq}
The set $\S_{c+}(\H)$ is a complete coanalytic subset of $\S(\H)$.
\end{thm}

\begin{proof}
    Consider
the diagonal map $\triangle\colon \I \hookrightarrow \H$, $x\mapsto (x, x, ,\ldots)$. This induces a function
$\tilde{\triangle}:S(\I) \hookrightarrow  S (\H)$ given by $A\mapsto \triangle(A)$. 

Note that $G\curvearrowright A$ is conjugate to $G\curvearrowright \triangle(A)$. Hence, $G\curvearrowright A$ has sofic CPMD if and only if  $G\curvearrowright \triangle(A)$ has sofic CPMD.

Let $\psi\colon K(\I)\rightarrow \S(\mathrm{\I})$ be the  function in the statement of Proposition \ref{cor:cpmd iff A countable}. Define $\psi_2=\tilde{\triangle}\circ \psi    \colon K(\I)\rightarrow \S(\mathrm{\H})$ such that $\psi_2(B)$ has sofic CPMD if and only if $B$ is countable.
As in the proof of Theorem \ref{thm:ccq}, by Theorem~\ref{hurewicz} and Proposition~\ref{prop:coanalyticbasic} we conclude that $\S_{c+}(Q)$ is  a complete coanalytic subset. 
\end{proof}

In \cite{darjilocal}, the authors proved that the family of interval maps with completely positive entropy is complete coanalytic. While some ideas of the proof are shared, there are some clear differences. The Polish spaces where the results are obtained are distinct (in each situation the most natural is used). In this paper, we use a space of shift spaces equipped with the Hausdorff metric that compares them as subsets. In the other paper, the topology of the Polish space is given by the sup-norm for continuous functions on a fixed space. Furthermore, the examples in \cite{darjilocal} have finite entropy and hence zero mean dimension.

Connections between topological entropy (and entropy pairs) and independence or interpolation have appeared in several papers \cite{huang2006local}\cite{kerr2009combinatorial}. It is thus natural to wonder if positive mean dimension is related to some notion of interpolation or independence.

\bibliographystyle{alpha}
\bibliography{universal_bib}

\Addresses

\end{document}